\documentclass[reqno,10pt]{amsart}

\numberwithin{equation}{section}

\usepackage{amssymb}
\usepackage{amsmath}
\usepackage{amsbsy}
\usepackage{amscd}
\usepackage{amsthm}
\usepackage{amsfonts}
\usepackage{enumerate}

\usepackage{color}

\usepackage{ifpdf}
\ifpdf \usepackage[pdftex,pdfstartview=FitH,pdfpagemode=none,colorlinks,bookmarks,linkcolor=blue]{hyperref} \else  \usepackage[hypertex]{hyperref} \fi

\usepackage[nobysame,abbrev,alphabetic]{amsrefs}
\newcommand{\hide}[1]{}

\newtheorem{theorem}{Theorem}[section]

\newtheorem{lemma}[theorem]{Lemma}
\newtheorem{corollary}[theorem]{Corollary}
\newtheorem{definition}[theorem]{Definition}

\newtheorem{conjecture}[theorem]{Conjecture}

\newtheorem{proposition}[theorem]{Proposition}
\newtheorem{notation}[theorem]{Notation}

\newtheorem{hypothesis}[theorem]{Hypothesis}

\theoremstyle{definition}\newtheorem{remark}[theorem]{Remark}
\theoremstyle{definition}\newtheorem*{acknowledgments}{Acknowledgments}

\newcommand{\cA}{\mathcal{A}}
\newcommand{\cB}{\mathcal{B}}

\newcommand{\cF}{\mathcal{F}}

\newcommand{\cI}{\mathcal{I}}
\newcommand{\cJ}{\mathcal{J}}

\newcommand{\cN}{\mathcal{N}}

\newcommand{\cP}{\mathcal{P}}

\newcommand{\cS}{\mathcal{S}}

\newcommand{\cU}{\mathcal{U}}
\newcommand{\cV}{\mathcal{V}}

\newcommand{\bC}{\mathbb{C}}
\newcommand{\bD}{\mathbb{D}}

\newcommand{\bR}{\mathbb{R}}
\newcommand{\bZ}{\mathbb{Z}}
\newcommand{\bQ}{\mathbb{Q}}

\newcommand{\bN}{\mathbb{N}}

\newcommand{\bT}{\mathbb{T}}

\newcommand{\hg}{{\hat{g}}}

\newcommand{\tR}{{\widetilde{R}}}

\newcommand{\tM}{{\widetilde{M}}}
\newcommand{\tF}{{\widetilde{F}}}

\newcommand{\tf}{{\tilde{f}}}
\newcommand{\tg}{{\tilde{g}}}

\newcommand{\teta}{{\tilde{\eta}}}

\newcommand{\Exp}{\mathop{\mathbb{E}}}

\newcommand{\gog}{\mathfrak{g}}

\newcommand{\bfe}{\mathbf{e}}
\newcommand{\bfh}{\mathbf{h}}

\newcommand{\bfj}{\mathbf{j}}

\newcommand{\bfn}{\mathbf{n}}

\newcommand{\bfN}{\mathbf{N}}

\newcommand{\re}{\operatorname{Re}}

\newcommand{\Ad}{\operatorname{Ad}}

\newcommand{\di}{\mathrm{d}}

\newcommand{\id}{\mathrm{id}}

\newcommand{\ab}{\mathrm{ab}}

\newcommand{\Poly}{\mathrm{Poly}}

\renewcommand{\mod}{\mathrm{mod}\ }

\newcommand{\onto}{\xymatrix{\ar@{>>}[r]&}}
\newcommand{\da}[4]{\xymatrix{#1 \ar@<.5ex>[r]^{#2} \ar@<-.5ex>[r]_{#3} & #4}}

\newcounter{subconst}[subsection]

\newcounter{const}

\newcounter{CONST}
\renewcommand{\theCONST}{C_{\arabic{CONST}}}

\newcommand{\lCONST}[1]{\refstepcounter{CONST}\label{#1}\theCONST}

\newcounter{BONST}
\renewcommand{\theBONST}{B_{\arabic{BONST}}}

\newcommand{\lBONST}[1]{\refstepcounter{BONST}\label{#1}\theBONST}

\begin{document}

\title{M\"obius Disjointness for Nilsequences Along Short Intervals}
\author[X. He]{Xiaoguang He}
\address{\newline Shandong University, Jinan, Shandong 250100, China\newline \rm hexiaoguangsdu@gmail.com}

\author[Z. Wang]{Zhiren Wang}
\address{\newline Pennsylvania State University, University Park, PA 16802, USA\newline \rm zhirenw@psu.edu}

\begin{abstract}
For a nilmanifold $G/\Gamma$, a $1$-Lipschitz continuous function $F$ and the M\"obius sequence $\mu(n)$, we prove a bound on the decay of the averaged short interval correlation $$\frac1{HN}\sum_{n\leq N}\Big|\sum_{h\leq H} \mu(n+h)F(g^{n+h}x)\Big|$$ as $H,N\to\infty$. The bound is uniform in $g\in G$, $x\in G/\Gamma$ and $F$.
\end{abstract}

\setcounter{page}{1}
\maketitle
{\small\tableofcontents}

\section{Introduction}
The M\"obius function $\mu : \mathbb{N} \rightarrow \{-1, 0, 1\}$ is defined as follows: $\mu(1)=1$, $\mu(n) =(-1)^k$ when $n$ is the product of $k$ distinct primes and $\mu(n) = 0$ otherwise. This is an important function since that $\sum\limits_{n\leq N}\mu(n)=o(N)$ is equivalent to prime number theorem, and  that
$\sum\limits_{n\leq N}\mu(n)=O_\varepsilon(N^{\frac{1}{2}+\varepsilon})$ for all $\varepsilon>0$ is equivalent to the Riemann Hypothesis.

The M\"obius Randomness Law, proposed in \cite{IK04}, suggests that reasonable sequences
$\xi(n)$ which have significant cancellations with $\mu(n)$, that is
$$\sum\limits_{n\leq N}\mu(n)\xi(n)=o(\sum\limits_{n\leq N}|\xi(n)|).$$
The M\"obius Disjointness Conjecture, of Sarnak \cite{S09}, expects to use observables from zero entropy topological dynamical systems as the sequence $\xi$.

\begin{conjecture} \label{ConjSarnak} (M\"obius Disjointness Conjecture, \cite{S09})
 Let $(X, T)$ be a topological dynamical system with zero topological entropy. Then 
\begin{equation}\label{EqConj}\lim_{N\to\infty}\frac{1}{N}\sum_{n=1}^{N}f(T^nx)\mu(n)=0, \forall f\in C(X), \forall x\in X.\end{equation}
\end{conjecture}

Here, a topological dynamical system is a pair $(X, T)$ consisting of a compact metric space $X$, and a continuous self-map $T: X\rightarrow X$. 

There have been many partial results on the Conjecture \ref{ConjSarnak}. For brevity we will simply refer to the recent comprehensive survey \cite{FKL18} for the progress in this area, and discuss only the historical developements that are more relevant to this paper.

The special case of Conjecture \ref{ConjSarnak} for circle rotations, has been known since 1937 due to Davenport's work \cite{D37}. Indeed, Davenport proved in \cite{D37} that for all $A>0$, \begin{equation}\label{EqDavenport}\sup_{\alpha\in\bR}\Big|\frac1N\sum_{n\leq N}e(\alpha n)\Gamma)\Big|\ll_A \log^{-A} N.\end{equation} Here $e(u)=e^{2\pi i u}$.

An important extension to this class is the nilsystems, namely tranlations $x\to g.x$ on a compact nilmanifold $G/\Gamma$. Such systems are particularly important because of their close relationship to multiple ergodic averages. Functions of the form $n\to f(g^n.x)$ cover all the polynomial and bracket polynomial phases. It was known, as a special case of Ratner's Theorem \cite{R91} and its discrete version by Shah \cite{S98}, that every trajectory of such a translation always equidistributes to the union of finitely many translated copies of a closed sub-nilmanifold. This property was extended by Leibman \cite{L05} to polynomial orbits in nilmanifolds (see Definition \ref{DefPoly} for the definition).

 M\"obius disjointness along orbits of nilsystems, or more generally polynomial orbits, was established by Green and Tao \cite{GT12b} in the following form:
\begin{equation}\label{EqGT12}\sup_{g,F}
\Big|\frac1N\sum_{n\leq N}\mu(n)F(g(n)\Gamma)\Big|\ll_{m,A}R^{-O_{m,A}(1)}\log^{-A} N,\end{equation} where the supremum is taken over all polynomial functions $g:\bZ\to G$ with respect to a given nilpotent filtration $G_\bullet$ and all functions $F:G/\Gamma\to\bC$ that are $1$-Lipschitz. Here $m=\dim G$, and the parameter $R$ records the rationality of the pair $(G_\bullet,\Gamma)$ (see Section \ref{SecNil} for related definitions).

Green-Tao's proof was based on their accompaying paper \cite{GT12a}, which effectivized Leibman's Theorem by describing quantitatively how fast a trajectory equidistributes to a subnilmanifold of $G/\Gamma$. This was then applied to joinings of two orbits of the forms $\{g(pn)\Gamma\}$ and $\{g(qn)\Gamma\}$. Combined with Vaughan's Identity \cite{V77}, which is a modern form of the Vinogradov bilinear method, such estimates lead to the orthogonality to the M\"obius function.

Another strengthening to Davenport's estimate \eqref{EqDavenport} was achieved in the recent breakthrough papers of Matom\"aki-Radziwi\l{}\l{} \cite{MR16} and Matom\"aki-Radziwi\l{}\l{}-Tao \cite{MRT16} on averages of non-pretentious multiplicative functions along short intervals. As a consequence, they proved in \cite{MRT16}  that for all real-valued 1-bounded multiplicative functions, which in particular include the M\"obius and Liouville functions, \begin{equation}\label{EqMRT}
  \sup_{\alpha \in \bR} \sum_{n\leq N} \left| \sum_{h\leq H} \mu(n+h) e(\alpha (n+h)) \right|\ dx \ll \left( \frac{\log\log H}{\log H} + \frac{1}{\log^{1/700} N} \right) HN.
\end{equation} Such estimates were used to prove an averaged form of the Chowla Conjecture in \cite{MRT16}, as well as the  logarithmically averaged Chowla and Elliott Conjectures for correlations with either 2 or an odd number of components by Tao \cite{T16} and Tao-Ter\"{a}v\"{a}inen \cite{TT17}. The theorems in  \cite{MR16} and  \cite{MRT16} have also yielded many applications to Conjecture \ref{ConjSarnak}, especially to dynamical systems with strong quasi-periodic behavior (see the survey \cite{FKL18}). They were also used in Frantzikinakis-Host's proof  \cite{FH18} of logarithmically averaged Sarnak Conjecture for ergodic weights. For most of these applications, it is essential to have a uniform decay rate in \eqref{EqMRT} that is independent of the choice of $\alpha$.

It is natural to seek a further strengthening to \eqref{EqDavenport} that combines the theorems of Green-Tao \eqref{EqGT12} and Matom\"aki-Radziwi\l{}\l{}-Tao \eqref{EqMRT}, namely a quantitative bound to M\"obius disjointness along short intervals for nilsequences. This is the purpose of the current paper.  This question is especially interesting because, as remarked in \cite{T16}*{p34}, short interval correlations between multiplicative functions and higher step nilsequences would be useful in the study of logarithmicall averaged Chowla and Elliott conjectures of higher order correlations.

Previously in this direction, Flaminio, Fr\polhk aczek, Ku\l aga-Przymus, and Lema\'{n}czyk \cite{FFKL19} proved that:
if  $\varphi$ is an ergodic unipotent affine automorphism of a compact nilmanifold $G/\Gamma$ and $x\in G/\Gamma$, $F\in C^0(G/\Gamma)$, then:
\begin{equation}
  \frac1N\sum_{N\leq n<2N}\left|\frac1H\sum_{h\leq H}\mu (n+h)
   F(\varphi^{n+h}(x))\right|\to 0
\end{equation}
  as $ H\rightarrow\infty $ and~$ N/H\rightarrow\infty$. Similar results were also shown for polynomial phases by El Abdalaoui-Lema'nczyk-de la Rue in \cite{ELD17}. The proofs purely relies on a minor arc argument and uses the bilinear method in the form of the K\'atai-Bourgain-Sarnak-Ziegler criterion \cites{K86, BSZ13}. The decay estimates in \cite{FFKL19} and \cite{ELD17} are not effective as the dynamics becomes highly quasi-periodic.

The result in this paper produces a uniformly effective bound without requiring ergodicity.

It should also be noted that without the extra average in $N$, non-trivial bounds on $\left|\frac1H\sum_{h\leq H}\mu (n+h)
   f(n+h)\right|$ were obtained in the works of Zhan \cite{Z91}, Huang \cite{H15, H16} and Matom\"aki-Shao \cite{MS19} when $f$ is a polynomial phase and $H\gg n^\theta$ for some given $\theta\in(0,1)$. ($\theta=\frac23$ in \cite{MS19}).

Our main theorem is:

\begin{theorem}\label{ThmMain} Suppose $G$ is a connected, simply connected $m$-dimensional nilpotent Lie group and $\Gamma\subset G$ is a lattice. Then there exists $H_0=H_0(G,\Gamma)>0$ and $\epsilon_0=\epsilon_0(m)>0$, such that:

For all $H,N\in\bN$ satisfying $H>H_0$ and $(\log N)^\frac12>\log H$, and $\epsilon\in(\frac{\log\log H}{\log H},\epsilon_0)$, there exists a set $\cS\in[N]$, whose construction depends only on $H$, $N$ and $\epsilon$, such that \begin{equation}\label{EqThmMainSSize}N-\#\cS\ll_m\epsilon N,\end{equation} and
\begin{equation}\label{EqThmMainS}\sup_{\substack{\|F\|_{G/\Gamma}\leq 1\\g\in G, x\in G/\Gamma}}\frac1{HN}\sum_{n\leq N}\Big|\sum_{h\leq H} 1_\cS\mu(n+h)F(g^{n+h}x)\Big|
\ll_mH^{-\epsilon}+\delta(H^\epsilon,N).\end{equation}

Here, the implied constants depend only on $m$. $\|F\|_{G/\Gamma}$ stands for the Lipschitz norm of a function $F$ on $G/\Gamma$.  The construction of the error function $\delta(\cdot,\cdot)>0$ is independent of all the parameters here, and it satisfies $\displaystyle\lim_{N\to\infty}\delta(a,N)=0$ for all $a>0$.

In particular, \begin{equation}\label{EqThmMain}\sup_{\substack{\|F\|_{G/\Gamma}\leq 1\\g\in G, x\in G/\Gamma}}\frac1{HN}\sum_{n\leq N}|\Big|\sum_{h\leq H} \mu(n+h)F(g^{n+h}x)\Big|
\ll_m\epsilon+H^{-\epsilon}+\delta(H^\epsilon,N).\end{equation}\end{theorem}

The Lipschitz norm of $F$ needs to be defined using a particular Mal'cev basis of the Lie algebra of $G$ that is compatible with $\Gamma$. For details, see \eqref{DefLip}.

By taking $\epsilon=\frac{\log\log H}{\log H}$, the following corollary immediately follows:

\begin{corollary}\label{CorMain} Suppose $G$ is a connected, simply connected $m$-dimensional nilpotent Lie group and $\Gamma\subset G$ is a lattice. Then there exists $H_0=H_0(G, \Gamma)>0$, such that:

For all $H,N\in\bN$ with $H>H_0$ and $(\log N)^\frac12>\log H$, \begin{equation}\label{EqThmMainS}\sup_{\substack{\|F\|_{G/\Gamma}\leq 1\\g\in G, x\in G/\Gamma}}\frac1{HN}\sum_{n\leq N}\Big|\sum_{h\leq H} \mu(n+h)F(g^{n+h}x)\Big|
\ll_m\frac{\log\log H}{\log H}+\delta(\log H,N),\end{equation} where the implied constant and the construction of the error function $\delta(\cdot,\cdot)>0$ depends only on $m$, and $\delta$ satisfies $\displaystyle\lim_{N\to\infty}\delta(a,N)=0$ for all $a>0$. \end{corollary}

In particular, in the settings of Corollary \ref{CorMain}, \begin{equation}\lim_{H\to\infty}\frac1H\limsup_{N\to\infty}\frac1N\sum_{n\leq N}\Big|\sum_{h\in H} \mu(n+h)F(g^{n+h}x)\Big|=0,\end{equation} uniformly for all $g\in G$, $x\in X$ and functions $F:G/\Gamma\to\bC$ from a given uniformly Lipschitz family.

\begin{remark}Theorem \ref{ThmMain} and Corollary \ref{CorMain} also hold if $\mu$ is replaced by the Liouville function $\lambda$, or any multiplicative function $\beta$ that is non-pretentious in the sense $M(\beta\chi,X)\to\infty$ as $X\to\infty$ for all Dirichilet characters $\chi$. For the definition of the quantitiy $M(\cdot,X)$, see Definition \ref{DefM}. A more precise version of Theorem \ref{ThmMain}, specifying how $\delta(H^\epsilon, N)$ depends on the functions $M(\beta\chi,\cdot)$, will be given in Theorem \ref{ThmMinorMajor}.\end{remark}

\begin{remark}Theorem \ref{ThmMinorMajor}, and thus Theorem \ref{ThmMain} and Corollary \ref{CorMain}, is actually valid for all polynomial sequnces $\{g(n,h)\Gamma\}$ in $G/\Gamma$ in lieu of $\{g^{n+h}x\}$. This in particular covers orbits of unipotent affine automorphisms as in \cite{FFKL19}.\end{remark}

We now outline the organization of the paper. The strategy in our proof mixes those from \cite{GT12b} and \cite{MRT16}. The main issue is that, while it is known by \cite{GT12a} that when $H$ is sufficiently large, each individual short range orbit $\{g^{n+h}x\}_{1\leq h\leq H}$ in $G/\Gamma$ should equidistribute well in a subnilmanifold $Y_n$, in order to apply the bilinear method, it is necessary to know that the equidstribution behaviors display a similar pattern in $Y_n$ and $Y_{n'}$ when $pn\approx p'n'$ for a pair of bouned prime numbers $p'$, $q'$. It is for this reason that we choose to view $g(n+h)$, where $g$ is a polynomial in one variable, as a polynomial $g(n,h)$ in two variables $n$ and $h$. After introducing the background notions in Section \ref{SecNil}, in Section \ref{SecLeibman} we derive a variation of Green-Tao's quantitative version of Leiman's Theorem that better adapts to our situation. Namely, we show that when $N$ and $H$ are both sufficiently large, $\{g(n,h)\Gamma\}_{1\leq h\leq H}$ is equidistributed in some $Y_n$ for a typical $n\leq N$, and the equidistrbution patterns in all such $Y_n$'s are correlated to each other. Section \ref{SecBilinear} sets up the bilinear methods scheme and separates the estimate into minor and major arcs along each short interval. In the major arc part (Section \ref{SecMajor}), the Matom\"aki-Radziwi\l{}\l{}-Tao estimate can be applied as the correspondence $n\to Y_n$ is periodic. In the minor arc part (Section \ref{SecMinorI}), we use Lemma \ref{LemSqFree} to replace the bilinear sum in \cite{MRT16}, which becomes a sum of 4-fold products after applying Cauchy-Schwarz and would get too complicated for nilsequences, with one that consists of 2-fold products recording the correlations between short orbits of the form $\{g(n,p(h+r))\}$ and $\{g(n',p'(h+r'))\}$ where $pn\approx p'n'$. The bound of such correlations, for all but a small portion of choices of $(n,n',p,p')$, will be given by Proposition \ref{PropJoining} and proved in Section \ref{SecMinorII} using the aforementionned correlation among equidistribution patterns. Finally, Section \ref{SecFinal} merges the minor and major arcs and fixes appropriate parameters to conclude the proof.

\begin{notation}\label{NotationMain}In this paper:
\begin{itemize}
\item $X=O_Y(Z)$ or $X\ll_Y Z$ means that $\frac XZ$ is bounded by a constant that depends only on $Y$.
\item Working under  Hypothesis \ref{HypMain}, we shall assume by default that the implicit constant $Y$ depend on the degree $d$ of the filtration and the dimension $m$ of the nilmanifold, without including $m$, $d$ in the subscript. For example, $O_A(1)$ will actually stand for $O_{A,m,d}(1)$.
\item $[N]$ stands for the interval of integers $\{1,\cdots,N\}$.
\item In the remainder of this paper, many implicit constants $O(1)=O_{m,d}(1)$ will appear. For simplicity, we will use a common constant $\lCONST{CONSTImplicit}=O_{m,d}(1)\geq 1$ that is large enough for all these purposes. Similarly, from now on the notation $\ll$ will always stand for $\ll_{m,\epsilon}$.
\item For $\alpha\in\bR$, $\|\alpha\|_{\bR/\bZ}$ denotes $\max_{k\in\bZ}|\alpha-k|$.
\end{itemize}\end{notation}

\begin{acknowledgments}A large part of this research was done while X.H. was visiting Pennsylvania State University during the 2017-2018 academic year. X.H. thanks the financial support (No. 201706220146) from China Scholarship Council and the hospitiality of Pennsylvania State University that made the visit possible. Z.W. was supported by NSF grants DMS-1501095 and DMS-1753042.

We thank Wen Huang for helpful comments.\end{acknowledgments}

\section{Background on sequences in nilmanifolds}\label{SecNil}

In this section, we quickly collect all the facts and notions that we will need from Green-Tao's papers \cite{GT12a}*{\S 1, \S2 \& \S A} and \cite{GT12b}*{\S 3}.

A connected, simply connected Lie group $G$ is {\bf nilpotent} if it has a nilpotent {\bf filtration} $G_\bullet$, i.e. a descending sequnce of groups $G=G_1\supseteq G_2\supseteq\cdots\supseteq G_d\supseteq G_{d+1}=\{e\}$ such that \begin{equation}\label{EqFiltration}[G,G_{i-1}]\subseteq G_i, \forall i\geq 2.\end{equation} This actually implies $[G_i,G_j]\subseteq G_{i+j}$ for all $i, j\geq 1$. The number $d$ is the {\bf degree} of the filtration $G_\bullet$.  The {\bf step} $s$ of $G$ is the degree of the lower central filtration defined by $G_{i+1}=[G,G_i]$.

For all $i\geq d+1$, we will adopt the convention that $G_i=\{e\}$.

Denote by $\gog_i$ the Lie algebra $G_i$, then $\gog_\bullet=\{\gog_i\}$ is a {\bf filtration of Lie algebras}, i.e. $[\gog,\gog_i]\subseteq\gog_{i+1}$, if and only if $G_i$ is a filtration.

A connected, simply connected nilpotent Lie group $G$ has a lattice $\Gamma$ if and only if it has an algebraic structure defined over $\bQ$. In this case, for a connected Lie subgroup $H$ of $G$, $H$ is an algebraic subgroup defined over $\bQ$ if and only if $H\cap\Gamma$ is a lattice of $H$. A lattice $\Gamma$ must be cocompact, and the compact quotient $G/\Gamma$ is called a {\bf nilmanifold}.

\hide{
For a normal subgroup $H$ of $G$, one can define a fiber product by \begin{equation}\label{EqFiberProd}G\times_HG=\{(g,g')\in G\times G: g'\in gH\}.\end{equation}

When $G$ is $d$-step nilpotent, consider the group $G^\square=G\times_{G_2}G\subseteq G^2$, which is again a $d$-step connect nilpotent Lie group as the sequence $G_\bullet^\square=\{G_i\times_{G_{i+1}}G_i\}_{i=1}^{d+1}$ is a filtration of $G^\square$. (\cite{GT12a}*{Proposition 7.2}).

The top entry in $G_\bullet^\square$ is $G_d^\square=G_d\times_{G_{d+1}}G_d=\{(g,h)\in G_d^2: gh^{-1}\in G_{d+1}\}=\{e\}\}=\{(g,g):g\in G_d\}$. The quotient group $\overline{G^\square}=G^\square/G_d^\square$ is thus $d-1$ step nilpotent, naturally equipped with the filtration $\overline{G_\bullet^\square}=\{\overline{G_i^\square}\}_{i=1}^d$, where $\overline{G_i^\square}=G_i^\square/G_d^\square$.
}

A basis $\cV=\{V_1,\cdots, V_m\}$ of $\gog$ is {\bf $R$-rational} if the structure constants $c_{ijk}$ in the Lie bracket relations $[V_i,V_j]=\sum_kc_{ijk}V_k$ are rational numbers whose heights are bounded by $R$. Recall that the height of a rational number $\frac ab$ is $\max(|a|,|b|)$ when $a$, $b$ are coprime. For nilmanifolds $G/\Gamma$, $G$ always has a rational basis. A special kind of rational basis, {\bf Mal'cev basis}, was defined in \cite{M49}. A rational basis $\cV=\{V_1,\cdots, V_m\}$ is a Mal'cev basis adapted to $(G_\bullet, \Gamma)$ if it satisfies the following properties in \cite{GT12a}*{Def. 2.1}:\begin{enumerate}[(i)]
\item $\{V_j,V_{j+1},\cdots,V_m\}$ spans an ideal of $\gog$ for all $0\leq j\leq m$;
\item For each $1\leq i\leq d$ and $m_i=\dim G_i$, the Lie algebra $\gog_i$ of $G_i$ is the linear span of  $\{V_{m-m_i+1},V_{m-m_i+2},\cdots,V_m\}$;
\item There is a diffeomorphism $\psi_\cV:G\to\bR^m$ determined by $$\psi_\cV\Big(\exp(\omega_1V_1)\cdots\exp(\omega_mV_m)\Big)=(\omega_1,\cdots,\omega_m);$$
\item In the coordinate system $\psi_\cV$, $\Gamma=\psi_\cV^{-1}(\bZ^m)$.
\end{enumerate}

When $G$ has a lattice $\Gamma$, there is always a Mal'cev basis adapted to the lower central filtration. In the coordinate system given by $\psi_\cV$, the set $\psi_\cV^{-1}([0,1)^m)$ will be a fundamental domain of the projection $G\to G/\Gamma$.

In the sequel, we will always assume that $G/\Gamma$ has a Mal'cev basis $\cV$ adapted to $(G_\bullet, \Gamma)$ for some filtration $G_\bullet$, and fix the tuplet $(G, G_\bullet, \Gamma, \cV)$. In this case, every $G_i$ is a rational subgroup of $G$, and $\Gamma_i=G_i\cap\Gamma$ is a lattice of $G_i$.

The nilmanifold $G/\Gamma$ has a tower structure of principal torus bundles $$G/\Gamma=G/G_{d+1}\Gamma\to G/G_d\Gamma\to\cdots\to G/G_2\Gamma\to G/G_1\Gamma=\{\text{pt}\},$$ where $G/G_{i+1}\Gamma$ is a principal $G_i/G_{i+1}\Gamma$-bundle over $G/G_i\Gamma$. Remark that here $G_i/G_{i+1}\Gamma\cong\bT^{m_i-m_{i+1}}$ is the quotient of the abelian Lie group $G_i/G_{i+1}\cong\bR^{m_i-m_{i+1}}$ by the lattice generated by the projections of $V_{m-m_i+1},\cdots, V_{m-m_{i+1}}$.

\hide{Since $G/\Gamma$ is principal $G_d/\Gamma_d$-fiber over the quotient nilmanifold $G/G_d\Gamma$ and $G_d/\Gamma_d$ can be identified with $\bT^{m_d}$, we can decompose a Lipschitz continuous function $F:G/\Gamma\to\bC$ as a Fourier series $\sum_{\xi\in\bZ^{m_d}}\widehat F_\xi(x)$ along the vertical fibers, where $\widehat F_\xi(x)=\int_{G_d/\Gamma_d}e^{-2\pi i\xi\cdot z}F(zx)\di z$. Then $F_\xi(zx)=e^{2\pi i\xi\cdot z}F_\xi(x)$ for all $z\in\bT^{m_d}$ and $x\in G/\Gamma$.

\begin{definition}$F$ is said to {\bf has vertical oscillation of frequency $\xi$} if $F=F_\xi$.\end{definition}

Remark that if $G/\Gamma$ has a Mal'cev basis $\cV$ adapted to $G_\bullet$. Then $\Gamma^\square=G^\square\cap\Gamma^2$ is a lattice in $G^\square$, because $G^\square$ is a rational subgroup of $G^2$. More generally, for every $i$, $\Gamma_i^\square=G^\square\cap\Gamma_i^2$  is a lattice in $G_i^\square$. And $\overline{\Gamma^\square}=\Gamma^\square/\Gamma_d^\square$ is a lattice in $\overline{G^\square}$.
}

A vector $v\in\gog$ is {\bf an $R$-rational combination} of elements in $\cV$ if $v=\sum v_jV_j$ where the $v_j$'s are rational numbers of height bounded by $R$. A subgroup $H\subseteq G$ is {\bf $R$-rational} with respect to $\cV$ if its Lie algebra has a basis consisting of such $R$-rational combinations.

The Mal'cev basis $\cV$ induces a right invariant metric $\di_G$ on $G$, which is the largest metric such that $\di(x,y)\leq|\psi_\cV(xy^{-1})|$ always holds. Actually, this in turn induces a metric $\di_{G/\Gamma}$ on $G/\Gamma$. For functions $F:G/\Gamma\to\bC$, $\|F\|$ will denote the Lipschitz norm \begin{equation}\label{DefLip}\|F\|=\|F\|_{C^0}+\sup_{n\neq y}\frac{|F(x)-F(y)|}{\di_{G/\Gamma}(x,y)}\end{equation} with respect to $\di_{G/\Gamma}$. We will also write $\|F\|_{G/\Gamma}$ instead, when it becomes necessary to emphasize that the distance is determined by the Mal'cev basis of $G/\Gamma$.

The nilpotent Lie group $G$ is unimodular, and $G/\Gamma$ has a unique left-invariant probability measure. The notation $\int_{G/\Gamma}$ will refer to the average with respect to this measure.

Since $G/[G,G]$ is abelian and the commutator subgroup $[G,G]$ is a rational subgroup, $(G/\Gamma)/([G,G]/([G,G]\cap\Gamma))=G/[G,G]\Gamma$ is a quotient torus of the connected abelian Lie group $G/[G,G]\cong\bR$, called the {\bf horizontal torus with respect to $G_\bullet$} of $G/\Gamma$.

\begin{definition}\cite{GT12a}*{Definition 2.6} A {\bf horizontal character} is a continuous additive homomorphism $\eta:G/[G,G]\Gamma\to\bR/\bZ$. We remark that $\eta$ can also be viewed as a continuous group homomorphism $\eta: G\to\bR/\bZ$ that vanishes on the subgroup $[G,G]\Gamma$.\end{definition} Using the coordinate representation $\psi_\cV$,  there exists an integer vector $a\in\bZ^m$, supported on the first $m-m_2$ coordinates, such that \begin{equation}\label{EqHorCharForm}\eta(g)=a\cdot\psi_\cV(g) (\mod\bZ).\end{equation} The {\bf modulus} $|\eta|$ of $\eta$ is defined to be $|a|$. Note $\eta$ is trivial if and only if $|\eta|=0$. By abusing notation, we shall also denote by $\eta$ the linear functional $\eta(v)= a\cdot v$ on $\bR^m\cong\gog$.

\begin{definition}\label{DefCNorm}For a polynomial function $f:[N]\to\bR/\bZ$ of degree at most $d$, $f$ can be written as $f(n)=\sum_{i=0}^d\alpha_i\binom{n}{i}$. The {\bf $C^\infty([N])$-norm} of $f$ is given by $$\|f\|_{C^\infty([N])}=\max_{i=0}^dN^i\|\alpha_i\|_{\bR/\bZ}.$$\end{definition}

\begin{lemma}\label{LemCoeffBound}\cite{GT12b}*{Lemma 3.2} If $f(n)=\sum_{i=0}^d\beta_in^i$, then there is an integer $D=O_d(1)$ such that $\|D\beta_i\|_{\bR/\bZ}\ll_d N^{-i}\|f\|_{C^\infty[N]}$ for all $i=0,\cdots,d$.\end{lemma}

\begin{lemma}\label{LemCNormBound}\cite{GT12a}*{Lemma 4.5} Suppose $f(n)=\sum_{i=0}^d\beta_in^i$, $\delta\in(0,\frac12)$, $\epsilon\in(0,\frac\delta2)$. If $f(n) (\text{mod } \bZ)$ belongs to an interval $I\subseteq\bR/\bZ$ of length $\epsilon$ for at least $\delta N$ integers $n\in[N]$. Then for some positive integer $D\ll_d\delta^{-O_d(1)}$, $\|Df (\text{mod }\bZ)\|_{C^\infty[N]}\ll_d\epsilon\delta^{-O_d(1)}$.\end{lemma}

For an integer vector $\bfN\in\bN^r$, write $[\bfN]=[N_1]\times\cdots\times [N_r]\subset\bZ^r$.

\begin{definition}\label{DefSmooth}\cite{GT12a}*{Definition 9.1} For a multiparatmeter finite sequence $\{g(\bfn)\}_{\bfn\in[\bfN]}$ in $G$ and an integer vector $\bfN\in\bN^r$, $g$ is said to be {\bf $(W,\bfN)$-smooth}, if for all $\bfn\in[\bfN]$, \begin{enumerate}\item $\di_G(g(\bfn),\id_G)\leq W$, 
\item $\di_G(g(\bfn),g(\bfn+\bfe_i))\leq\frac W{N_i}$ for all $i$, where $\bfe_i$ is the unit vector along the $i$-th coordinate direction.\end{enumerate}\end{definition}

If $g_1, g_2$ are both $(W,\bfN)$-smooth, and $W\geq R$, where the metric is induced by an $R$-rational Mal'cev basis, then $g_1g_2$ is $(W^{O(1)},\bfN)$ smooth.

\begin{definition}\label{DefRational} An element $g\in G$ is {\bf $R$-rational}, if there exists $1\leq r\leq R$ such that  $g^r\in\Gamma$. An element $z\in G/\Gamma$ is {\bf $R$-rational}, if $z=g\Gamma$ for some $R$-rational group element $g$.
\end{definition}

\begin{lemma}\label{LemRational}\cite{GT12a}*{Lemma A.11} Suppose the Mal'cev basis $\cV$ adapted to $(G_\bullet,\Gamma)$ is $R$-rational. With respect to $\cV$, if $g$ is $R$-rational then $\psi_\cV(g)\in\frac 1q\bZ^m$ for some $q\ll R^{O(1)}$. Conversely, if $\psi_\cV(g)\in\frac1R\bZ^m$ then $g$ is $R^{O(1)}$-rational. Moreover, the product of two $R$-rational elements is $R^{O(1)}$-rational.\end{lemma}

\begin{definition}\label{DefED}For  a finite arithmetic progression $\cA=\{qn+r\}_{n\in[N]}$ in $\bZ$, a finite sequence $\{x(n)\}_{n\in\cA}$ in $G/\gamma$ is said to be {\bf $\delta$-equidistributed} in $G/\Gamma$ if for all Lipschitz function $F$ on $G/\Gamma$, $$\left|\Exp_{n\in\cA}F(x(n))-\int_{G/\Gamma}F\right|\leq\delta\|F\|;$$
and it is {\bf totally $\delta$-equidistributed} in $G/\Gamma$  if the subsequence $\{x(n)\}_{n\in\cA'}$ is $\delta$-equidistributed in $G/\Gamma$ for all arithmetic progressions $\cA'\subseteq\cA$ of length at least $\delta N$.\end{definition}

\begin{lemma}\label{LemEtaTotalED}Suppose a Mal'cev basis $\cV$ adapted to $(G_\bullet,\Gamma)$ is $R$-rational where $R\geq 10$. Let $\eta$ be a non-trivial horizontal character of $G/\Gamma$, whose modulus $|\eta|$ is bounded by $R$ with respect to $\cV$. If for a polynomial  sequence $g\in\Poly(\bZ,G_\bullet)$ and $N\gg R$, $\|\eta\circ g\|_{C^\infty([N])}\leq R$, then $\{g(n)\Gamma\}_{n\in [N]}$ is not totally $(O(R))^{-1}$-equidistributed.\end{lemma}
\begin{proof}Because $\|\eta\circ g\|_{C^\infty([N])}\leq R$,
$\|\eta\circ g(n)-\eta\circ g(0)\|_{\bR/\bZ}\ll RnN^{-1}
$. This implies that for the the mapping $\teta(x)=\exp(2\pi i \eta(x))$ from $G/\Gamma$ to the unit circle in $\bC$, the values of $\teta(g(n))$ are within distance $\ll R\delta$ to each other for $0<n\leq\delta N$. Using the convention in Notation \ref{NotationMain}, one can assume that the implicit constant here is $\ref{CONSTImplicit}$. In particular,
\begin{equation}\begin{aligned}
\Big|\Exp_{0<n\leq \delta N}\teta(g(n)\Gamma)\Big|>1-\ref{CONSTImplicit}R\delta\geq\frac12,\end{aligned}\end{equation} if $\delta<\frac12\ref{CONSTImplicit}^{-1}R^{-1}$. Because $\eta$ is a non-zero character, $\teta$ has zero average on $G/\Gamma$. In addition, $\|\teta\|_{G/\Gamma}\leq 2\pi|\eta|\leq 2\pi R$.  It follows that the sequence $\{g(h)\Gamma\}_{h\in [H]}$ is not totally $\min(\frac12\ref{CONSTImplicit}^{-1}R^{-1},\frac1{4\pi}R^{-1})$-equidistributed in $G/\Gamma$.\end{proof}

\begin{lemma}\label{LemEDTotalED} If $\delta\in(0,1)$ and there exists an interval $\cA\subseteq[N]$ of length at least $\delta N$ such that $\{g(n)\}_{n\in\cA}$ is not $\delta$-equidistributed in $G/\Gamma$, then for some $N'\in[\frac{\delta^2}2N,N]$, $(g(n))_{n\in[N']}$ is not $\frac{\delta^2}2$-equidistributed in $G/\Gamma$.\end{lemma}

\begin{proof} One may write $\cA=\{N_1< n\leq N_2\}=[N_2]\setminus[N_1]$. Write $\theta_i=\frac{N_i}N$ and $\theta=\theta_2-\theta_1$, then $\theta\geq\delta$.

There exists a Lipschitz function $F$ on $G'/\Gamma'$ with $\int_{G/\Gamma}F=0$ such that
$$\left|\frac{\theta_2}{\theta}\Exp_{n\in[N_2]}F(g(n)\Gamma)-\frac{\theta_1}{\theta}\Exp_{n\in[N_1]}F(g(n)\Gamma)\right|=\left|\Exp_{n\in\cA}F(g(n)\Gamma)\right|>\delta\|F\|.$$

If $\theta_1\geq\frac{\delta^2}2$ and $\left|\Exp_{n\in[N_1]}F(g(n)\Gamma)\right|>\frac{\delta^2}2\|F\|$, then $N_1\geq\frac{\delta^2}2N$ and $(g(n))_{n\in[N_1]}$ is not $\frac{\delta^2}2$-equidistributed.

Otherwise, either $\theta_1<\frac{\delta^2}2$ or $\left|\Exp_{n\in[N_1]}F(g(n)\Gamma)\right|<\frac{\delta^2}2\|F\|$. In both cases, $$\left|\frac{\theta_1}{\theta}\Exp_{n\in[N_1]}F(g(n)\Gamma)\right|<\frac{\delta^2}2\|F\|,$$ 
and thus $$\left|\Exp_{n\in[N_2]}F(g(n)\Gamma)\right|\geq\left|\frac{\theta_2}{\theta}\Exp_{n\in[N_2]}F(g(n)\Gamma)\right|>\delta\|F\|-\frac{\delta^2}2\|F|\geq \frac\delta2\|F\|.$$ So $(g(n))_{n\in[N_2]}$ is not $\frac\delta2$-equidistributed. Moreover, $N_2\geq\theta N \geq\delta N$.
 \end{proof}

For a map $g:\bZ^r\to G$, the derivative along  $\bfh\in\bZ^r$ is
\begin{equation}\partial_\bfh g(\bfn)=g(\bfn+\bfh)g(\bfn)^{-1}.\end{equation}

\begin{definition}\label{DefPoly}A map $g:\bZ^r\to G$ is a {\bf polynomial map with respect to $G_\bullet$} if for all $i$ and $l_1,\cdots,l_i,n\in\bZ$, the $i$-th derivative $\partial_{l_1}\cdots\partial_{l_i}g(n)$ takes values in $G_i$. The set of
polynomial sequences with respect to $G_\bullet$ is noted by $\Poly(\bZ^r, G_\bullet)$.\end{definition}

The family of $\Poly(\bZ^r, G_\bullet)$ is known to be a group (Lazard \cite{L54}, Leibman \cites{L98,L02} and Green-Tao \cite{GT12a}). A description of $\Poly(\bZ^r, G_\bullet)$ was given in Leibman and Green-Tao's works:

\begin{lemma}\label{LemPolyCoeff}(\cite{L10}*{\S 4},\cite{GT12a}*{\S 6}) Suppose $\cV$ is a Mal'cev basis adapted to $(G_\bullet,\Gamma)$, then $g\in\Poly(\bZ^r,G)$ if and only if $\psi_\cV(g(\bfn))$ has the form
$$\psi_\cV(g(\bfn))=\sum_{\bfj\in\bZ_{\geq 0}^r}\omega_\bfj\binom{n_1}{j_1}\cdots\binom{n_r}{j_r},$$ where $\omega_\bfj\in\bR^m$ and $(\omega_\bfj)_i=0$ for all $i\leq m-m_{|\bfj|}$ with $|\bfj|=j_1+\cdots+j_r$.\end{lemma}

In particular, if $|\bfj|>d$, then $m_{|\bfj|}=0$ and thus $\omega_\bfj=0$.

In the rest of this paper we will work under the following work hypothesis

\begin{hypothesis}\label{HypMain}$G/\Gamma$ is an $m$-dimensional compact nilmanifold with a degree $d$ rational filtration $G_\bullet$, and $\cV$ is an $R_0$-rational Mal'cev basis adapted to $(G_\bullet, \Gamma)$, where $R_0>10$. Moreover, $g\in\Poly(\bZ^2,G_\bullet)$ is a polynomial map determined by coefficients $\{\omega_{j,k}\}_{j,k\in\bZ_{\geq 0}}$ as in Lemma \ref{LemPolyCoeff}. Let $R\geq R_0$ be a parameter to be determined later.
In particular, $\cV$ is also an $R$-rational Mal'cev basis  adapted to $(G_\bullet, \Gamma)$.\end{hypothesis}

The formula in Lemma \ref{LemPolyCoeff} writes in this case as:
\begin{equation}\label{EqPolyCoeff}\psi_\cV(g(n,h))=\sum_{\substack{j,k\geq 0\\j+k\leq d}}\omega_{jk}\binom{n}{j}\binom{h}{k},\end{equation} where $(\omega_{jk})_i=0$ for all $i\leq m-m_{j+k}$.

\section{Quantitaive factorization theorem for 2-parameter polynomials}\label{SecLeibman}

We now state Green-Tao's effectivization of a theorem of Leibman \cite{L05}, and deduce a variation of it that is refined to our situation.

\begin{proposition}\label{PropLeibmanGT} \cite{GT12a}*{Theorem 2.9} Suppose $G/\Gamma$ is an $m$-dimensional compact nilmanifold with a degree $d$ rational filtration $G_\bullet$, and $\cV$ is an $R$-rational Mal'cev basis adapted to $(G_\bullet, \Gamma)$ where $R\geq 10$. For $f\in\Poly(\bZ,G_\bullet)$, and $N\in\bN$ such that $N\gg  R^{O(1)}$, at least one of the following holds:\begin{enumerate}
\item\label{PropLeibmanGT1} either $\{f(n)\Gamma\}_{n\in[N]}$ is $ R^{-1}$-equidistributed in $G/\Gamma$;
\item\label{PropLeibmanGT2} or there exists a horizontal character $\eta$ of $G/\Gamma$ of modulus $|\eta|\leq  R^{O(1)}$ such that $\|\eta\circ f\|_{C^\infty([N])}\leq  R^{O(1)}$.\end{enumerate}\end{proposition}

\begin{corollary}\label{CorLeibman}In Proposition \ref{PropLeibmanGT}, one may replace in part \eqref{PropLeibmanGT1} the property ``$ R^{-1}$-equidistributed'' by ``totally $ R^{-1}$-equidistributed''.\end{corollary}
\begin{proof}Suppose $\{f(n)\Gamma\}_{n\in[N]}$ is  not totally $ R^{-1}$-equidistributed. There exist integers $0\leq a<b\leq  R$, and an interval $\cA\subseteq[\frac Nb]$ of length at least $ R^{-1}N$, such that the sequence $\{\tf(n)\Gamma\}_{n\in\cA}$  is not $ R^{-1}$-equidistributed, where $\tf(n)=f(bn+a)$.  By Lemma \ref{LemEDTotalED}, there exists $N'<N$ with $N'\geq \frac12 R^{-2}\cdot \#\cA\geq  R^{-O(1)}N$ such that $\{\tf(n)\Gamma\}_{n\in[N']}$  is not $ R^{-1}$-equidistributed. By Proposition \ref{PropLeibmanGT}, there exists a horizontal character $\eta$ such that $0<|\eta|< R^{O(1)}$ and $\|\eta\circ \tf\|_{C^\infty([N'])}\leq R^{O(1)}$. As $N'\geq  R^{-O(1)}N$, this implies that $\|\eta\circ \tf\|_{C^\infty([N])}\leq  R^{O(1)}$, which in turn implies by \cite{GT12a}*{7.10} that there is a positive integer $D\leq R^{O(1)}$ such that $\|D\eta\circ f\|_{C^\infty([N])}\ll  R^{O(1)}$. The corollary then follows after replacing $\eta$ with $D\eta$.\end{proof}

\begin{corollary}\label{CorChangeLattice}Suppose $G$ is an $m$-dimensional simply connected Lie group with a degree $d$ rational filtration $G_\bullet$, and $\Gamma_j$ is a lattice in $G$ for $j=1,2$ and $\cV_j$ is an $R$-rational Mal'cev basis adapted to $(G_\bullet, \Gamma_j)$. Assume in addition that elements in $\cV_2$ are $R$-rational combinations of elements in $\cV_1$.

For $f\in\Poly(\bZ,G_\bullet)$, and $N\in\bN$ such that $N\gg  R^{O(1)}$, if $\{f(n)\Gamma_1\}_{n\in[N]}$ is not totally $R^{-1}$-equidistributed in $G/\Gamma_1$, then $\{f(n)\Gamma_2\}_{n\in[N]}$ is not totally $R^{-O(1)}$-equidistributed in $G/\Gamma_2$.\end{corollary}
\begin{proof}By Corollary \ref{CorLeibman}, there is a non-trivial horizontal character $\eta$ of $G/\Gamma_1$, i.e. a character $G\to\bR/\bZ$ that annihilates $\Gamma_1$, of size $|\eta|_{\cV_1}\leq R^{O(1)}$ that satisfies $\|\eta\circ f\|_{C^\infty([N])}\leq R^{O(1)}$. Here the modulus $|\eta|_{\cV_1}\leq R^{O(1)}$ is measured in terms of the basis $\cV_1$. Because all elements of $\cV_2$ are $R$-rational combinations of those in $\cV_1$, by Lemma \ref{LemRational}, there is a positive integer $D\leq R^{O(1)}$ such that for all $\gamma\in \Gamma_2$,  $\gamma^D\in\Gamma_1$ and thus $D\eta(\gamma)=\eta(\gamma^D)=0$. Then $D\eta$ is a horizontal character of both $G/\Gamma_1$ and $G/\Gamma_2$ with $|D\eta|_{\cV_1}\leq R^{O(1)}$. Again, because all elements of $\cV_2$ are $R$-rational combinations of those in $\cV_1$, $|D\eta|_{\cV_2}\leq R^{O(1)}$.  After replacing $\eta$ with $D\eta$, one may assert that:

There exists a non-trivial horizontal character $\eta$ of $G/\Gamma_2$ such that $|\eta|_{\cV_2}\leq R^{O(1)}$ and $\|\eta\circ f\|_{C^\infty([N])}\leq R^{O(1)}$. By Lemma \ref{LemEtaTotalED}, $\{f(n)\Gamma_2\}_{n\in[N]}$ fails to be totally $R^{-O(1)}$-equidistributed.\end{proof}

The following is the refined statement  that we will need later, which deals with generic restrictions of a 2-parameter polynomial to one variable.

\begin{proposition}\label{PropLeibman}Under Hypothesis \ref{HypMain}, for $\tR\geq R$ and $N, H\in\bN$ such that $N,H\gg \tR^{O(1)}$, at least one of the following holds:\begin{enumerate}
\item\label{PropLeibman1} either $\{g(n,h)\Gamma\}_{h\in[H]}$ is totally $\tR^{-1}$-equidistributed in $G/\Gamma$ for all but $\tR^{-1}N$ values of $n\in[N]$;
\item\label{PropLeibman2} or there exists a horizontal character $\eta$ of $G/\Gamma$ of modulus $|\eta|\leq \tR^{O(1)}$ such that $\|\eta(\omega_{j,k})\|_{\bR/\bZ}\leq \tR^{O(1)}N^{-j}H^{-k}$ for all $j,k\geq 0$.
\end{enumerate}
\end{proposition}

\begin{proof}Assuming \eqref{PropLeibman1} fails, we try to establish  \eqref{PropLeibman2}. For more than $\tR^{-1}N$ values of $n\in[N]$,  $\{g(n,h)\Gamma\}_{h\in[H]}$ is not totally $\tR^{-1}$-equidistributed. For every such $n$, by Corollary \ref{CorLeibman} there is a horizontal character $\eta$ with $|\eta|\leq\tR^{O(1)}$ such that \begin{equation}\label{EqPropLeibman1}\|\eta\circ g(n,\cdot)\|_{C^\infty([H])}\ll \tR^{O(1)}.\end{equation}

Applying pigeonhole principle to the at least $\tR^{-1}N$ values of $n\in[N]$, there is a common $\eta$ with $0<|\eta|<\tR^{O(1)}$, such that \eqref{EqPropLeibman1} holds for at least  $\tR^{-O(1)}N$ choices of $n\in[N]$. By \eqref{EqPolyCoeff}, this implies:
$$\Bigg\|\sum_{\substack{j,k\geq 0\\j+k\leq d}}\binom{n}{j}\binom{\cdot}{k}\eta(\omega_{jk})\Bigg\|_{C^\infty([H])}\ll \tR^{O(1)},$$  which by Definition \ref{DefCNorm} means that
\begin{equation}\label{EqPropLeibman1}\Bigg\|\sum_{j=0}^{d-k}\binom{n}{j}\eta(\omega_{jk})\Bigg\|_{\bR/\bZ}\ll \tR^{O(1)}H^{-k},\ \forall k=0,\cdots,d.\end{equation}

As this inequality holds for $\tR^{-O(1)}N$ choices of $n\in[N]$, by Lemma \ref{LemCNormBound} there is a positive integer $D>0$ such that
$$\Bigg\|D\sum_{j=0}^{d-k}\binom{\cdot}{j}\eta(\omega_{jk})\Bigg\|_{C^\infty([N])}\ll \tR^{O(1)}H^{-k}\cdot \tR^{O(1)}=\tR^{O(1)}H^{-k},\ \forall k=0,\cdots,d.$$ In other words,
\begin{equation}\label{EqPropLeibman2}\|D\eta(\omega_{jk})\|_{\bR/\bZ}\ll \tR^{O(1)}H^{-k}N^{-j}, \forall k,j\geq 0\text{ such that }k+j\leq d.\end{equation}
This is exactly the desired conclusion after replacing $\eta$ with $D\eta$.
\end{proof}

\begin{lemma}\label{LemLeibmanDecomp}If Case \ref{PropLeibman}.\eqref{PropLeibman2} holds in Proposition \ref{PropLeibman}, then there is a decomposition $g=\epsilon g'\gamma$ with $\epsilon,g',\gamma\in\Poly(\bZ^2,G)$ such that:
\begin{enumerate}
\item\label{LemLeibmanDecomp1} $\epsilon$ is $(\tR^{O(1)},(N,H))$-smooth;
\item\label{LemLeibmanDecomp2} $\eta\circ g'=0$ while regarding $\eta:G/\Gamma\to\bR/\bZ$ as a morphism from $G$ to $\bR$;
\item\label{LemLeibmanDecomp3} $\gamma(n,h)$ is $\tR^{O(1)}$-rational for all $n,h\in\bZ$. \end{enumerate}\end{lemma}
\begin{proof}The proof is the same as that of \cite{GT12a}*{Lemma 9.2} except that we are not reducing to the case $g(0)=\id$. For completeness, we give a sketch.

For all integer pairs $j,k\geq 0$ with $j+k\leq d$, choose $u_{jk}\in\bR^m$ such that $\eta(u_{jk})\in\bZ$ and $|\omega_{jk}-u_{jk}|\ll \tR^{O(1)}N^{-j}H^{-k}$, and $v_{jk}\in\bQ^m$ such that $\eta(u_{jk})=\eta(v_{jk})$, where $
\eta$ is regarded as an $\bR$-valued linear functional from $\bR^m\cong\gog$. This can be done while requiring that $(u_{jk})_i=(v_{jk})_i=0$ for all $i\leq m-m_{j+k}$. Furthermore, one can require that $v_{j,k}$ is from $(\frac1D\bZ)^m$  for some integer $1\leq D\leq \tR^{O(1)}$.

Then define $\epsilon$, $g'$ and $\gamma$ by
$$\psi_\cV(\epsilon(n,h))=\sum_{\substack{j,k\geq 0\\j+k\leq d}}(\omega_{jk}-u_{jk})\binom{n}{j}\binom{h}{k},\ \psi_\cV(\gamma(n,h))=\sum_{\substack{j,k\geq 0\\j+k\leq d}}v_{jk}\binom{n}{j}\binom{h}{k},$$ and $g'(n,h)=\epsilon(n,h)^{-1}g(n,h)\gamma(n,h)^{-1}$. Then by Lemma \ref{LemPolyCoeff}, $\epsilon$, $\gamma$ belong to $\Poly(\bZ^2,G_\bullet)$ and hence so does $g'$ as $\Poly(\bZ^2,G_\bullet)$  is a group.

By the bound on $|\omega_{jk}-v_{jk}|$, we know that for all $(n,h)\in[N]\times[H]$, $$|\psi_\cV(\epsilon(n+1,h))-\psi_\cV(\epsilon(n,h))|\ll\sum_{\substack{j\geq 1,k\geq 0\\j+k\leq d}}\tR^{O(1)}N^{-j}H^{-k}\cdot n^{j-1}h^k\ll \tR^{O(1)}N^{-1}$$ and similarly
$|\psi_\cV(\epsilon(n,h+1))-\psi_\cV(\epsilon(n,h))|\ll \tR^{O(1)}H^{-1}$. Moreover, $|\psi_\cV(\epsilon(0,0))|=|\omega_{00}-v_{00}|\ll \tR^{O(1)}$. These inqualities guarantee property \eqref{LemLeibmanDecomp1} for $\epsilon$ by \cite{GT12a}*{Lemma A.5}.

Property \eqref{LemLeibmanDecomp2} holds as $$\begin{aligned}&\eta(g'(n,h))\\=&\eta(g(n,h))-\eta(\epsilon(n,h))-\eta(\gamma(n,h))\\
=&\sum_{\substack{j,k\geq 0\\j+k\leq d}}\eta(\omega_{jk})\binom{n}{j}\binom{h}{k}-\sum_{\substack{j,k\geq 0\\j+k\leq d}}\eta(\omega_{jk}-u_{jk})\binom{n}{j}\binom{h}{k}-\sum_{\substack{j,k\geq 0\\j+k\leq d}}\eta(v_{jk})\binom{n}{j}\binom{h}{k}\\
=&0.\end{aligned}$$ 

Finally, it follows from Lemma \ref{LemRational} that $\gamma$ is $\tR^{(O(1)}$-rational. This also implies by \cite{GT12a}*{Lemma A.12} (or rather the natural multiparameter extension of it) that for some positive integer $q\ll (\tR^{O(1)})^{O(1)}\ll \tR^{O(1)}$, $\gamma(n,h)\Gamma$ is $q\bZ^2$-periodic. Thus we have property \eqref{LemLeibmanDecomp3}.\end{proof}

Using this, Green-Tao's factorization theorem \cite{GT12a}*{Theorems 1.19 \& 10.2} can be easily refined to the following:

\begin{theorem}\label{ThmFactorization}Under Hypothesis \ref{HypMain}, for $B\geq 1$, $N, H\in\bN$ such that $N,H\gg R^{O(1)}$,  there exists an integer $W\in[R,R^{O(B^m)}]$, a $W$-rational subgroup $G'\subseteq G$, a $W$-rational Mal'cev basis $\cV'$ adapted to $(G'_\bullet,G'\cap\Gamma)$ consisting of $W$-rational combinations of vector in $\cV$, and a decomposition $g=\epsilon g'\gamma$ with $\epsilon,g',\gamma\in\Poly(\bZ^2,G_\bullet)$ such that:
\begin{enumerate}
\item\label{ThmFactorization1} $\epsilon$ is $(W,(N,H))$-smooth.
\item\label{ThmFactorization2} $g'$ takes value in $G'$. And, with respect to the metric induced by $\cV'$ on $G'/\Gamma'$, $\{g'(n,h)\}_{h\in[H]}$ is totally $W^{-B}$-equidistributed for all but at most $W^{-B}N$ values of $n\in[N]$;
\item\label{ThmFactorization3} $\gamma(n,h)$ is $W$-rational for all $n,h\in\bZ$. Moreover for some $1\leq q\leq W$, $\{\gamma(n,h)\Gamma\}_{(n,h)\in\bZ^2}$ is $q\bZ^2$-periodic.
\end{enumerate}
\end{theorem}

\begin{proof}We start with the squence $g(n)$ apply Proposition \ref{PropLeibman} with $\tR=R^B$. If Case  \ref{PropLeibman}.\eqref{PropLeibman1} holds, then the theorem is true for $G'=G$, $W=R$, $\epsilon(n,h)=\gamma(n,h)=\id$ and $g'=g$.

 If  Case  \ref{PropLeibman}.\eqref{PropLeibman2} holds for a non-trivial horizontal character $\eta_1$ of $G/\Gamma$ of norm $\ll \tR^{O(1)}$ and Lemma \ref{LemLeibmanDecomp} applies, yielding a decomposition $g=\epsilon_1g'_1\gamma_1$. In this case, let $G'_1=\ker_G\eta_1$ and $\Gamma'_1=G'_1\cap\Gamma$. Then $(G'_1)_\bullet=\{(G'_1)_i\}_{i\geq 0}=\{G'_1\cap G_i\}_{i\geq 0}$ is a filtration of $G'_1$. Notice that each $(G'_1)_i$ is a $\tR^{O(1)}$-rational subgroup. For $R_1=\tR^{O(1)}=R^{O(B)}$, by \cite{GT12a}*{Lemma A.10} $G_1$ has an $R_1$-rational Mal'cev basis $\cV_1$ adapted to $((G_1)_\bullet,\Gamma'_1)$ consisting of $R_1$-rational combinations of vector in $\cV$.

We then again to apply Proposition \ref{PropLeibman} with $\tR=R_1^B$, and apply Lemma \ref{LemLeibmanDecomp} if necessary, to the sequence $\{g'_1(n)\Gamma'_1\}$ in $G_1/\Gamma'_1$. The argument is iterated if Case  \ref{PropLeibman}.\eqref{PropLeibman2} holds in every step. So in the $k$-th step, we will apply Proposition \ref{PropLeibman} with $\tR=R_{k-1}^B$, and obtain, with $R_k=\big(R_{k-1}^B\big)^{O(1)}=(R_{k-1})^{O(B)}$:\begin{itemize}
\item a non-trivial horizontal charcter $\eta_k$ of $G'_{k-1}/\Gamma'_{k-1}$ of norm $\ll R_k$;
\item an $R_k$-rational Mal'cev basis $\cV_k$ adapted to $((G'_k)_\bullet,\Gamma'_k)$ consisting of $R_k$-rational combinations of vector in $\cV_{k-1}$, where $G'_k=\ker_{G_{k-1}}\eta_k$ and $(G'_k)_i=G'_k\cap G_i$;
\item a decomposition $g'_{k-1}=\epsilon_kg'_k\gamma_k$ in the group $\Poly(\bZ^2,(G_{k-1})_\bullet)$,
\end{itemize}
such that:\begin{itemize}
\item $\epsilon$ is $(R_k,(N,H))$-smooth with respect to the metric induced by $\cV_{k-1}$ on $G'_{k-1}$;
\item $g'_k$ takes value in $G'_k$, and thus $g'_k\in\Poly(\bZ^2,(G'_k)_\bullet)$;
\item $\gamma'_k$ is $R_k$-rational with respect to the Mal'cev basis $\cV_{k-1}$.
\end{itemize}

As $\dim G'_k$ strictly decreases, the process must stop at some $k\leq m$. This means Case \ref{PropLeibman}.\eqref{PropLeibman1} holds, i.e. $\{g'_k(n,h)\Gamma_k\}_{h\in[H]}$ is totally $R_k^{-B}$-equidistributed in $G'_k/\Gamma'_k$ for all but $R_k^{-B}N$ values of $n\in[N]$.

Write $g=\epsilon g'\gamma$ where $\epsilon=\epsilon_1\cdots\epsilon_k$, $g'=g'_k$ and $\gamma=\gamma_k\cdots\gamma_1$, $G'=G'_k$, $\cV'=\cV_k$ and $W=R_k$. Notice that since for each $j$, $\epsilon_j\in\Poly(\bZ^2,(G'_j)_\bullet)\subseteq \Poly(\bZ^2,G_\bullet)$ and $\Poly(\bZ^2,G_\bullet)$ is a group, $\epsilon\subseteq\Poly(\bZ^2,G_\bullet)$. Similarly $\gamma$ is in $\Poly(\bZ^2,G_\bullet)$ and so is $g'$.

It was shown above that the property \eqref{ThmFactorization2} in the theorem holds for $g'$. The properties \eqref{ThmFactorization1}  and  the $W$-rationality in \eqref{ThmFactorization3} follow in the same way as in the proof of \cite{GT12a}*{Theorem 10.2}, after replacing $W$ with $W^{O(1)}$ if necessary. Furthermore, by a multiparameter version of \cite{GT12a}*{Lemma A.12}, the 2-parameter sequence $\{\gamma'(n,h)\Gamma\}_{(n,h)\in\bZ^2}$ is $q\bZ^2$-periodic for some $q\ll W^{O(1)}$. Once again by replacing $W$ with $W^{O(1)}$, we obtain the property \eqref{ThmFactorization3} for $\gamma$.

Finally, remark that as $k\leq m$, $R_k\ll R^{O(B^m)}$ and $W\ll R_k^{O(1)}\ll R^{O(B^m)}$.\end{proof}

\section{Separation of major and minor arcs}\label{SecBilinear}

From now on, we work under Hypothesis \ref{HypMain}.

\begin{notation}\label{NotationBilinear}Suppose $\ref{CONSTImplicit}=O(1)$ is sufficiently large, and $\lBONST{BONSTTotalED}\geq 10\ref{CONSTImplicit}$. Let $N$, $H$, and $g$ be  as in Theorem \ref{ThmFactorization}, applied with $B=\ref{BONSTTotalED}$.  Also let $\epsilon$, $g'$, $\gamma$, $W$, $q$, $G'$ and $\cV'$ be as in the conclusion of the theorem. Without loss of generality, we may assume $R\geq 10$. In addition, after replacing the period $q$ with a multiple of it if necessary, we may assume $q\in (\frac W2, W]$. \end{notation}

Because $W\in [R, R^{O(\ref{BONSTTotalED}^m)}]$, we will fix a constant $\lCONST{CONSTWRange}=O_{m,d}(1)\geq 1$ and assume
\begin{equation}W\in [R, R^{\ref{CONSTWRange}\ref{BONSTTotalED}^m}].\end{equation}

Let $F: G/\Gamma\to\bC$ be a function with $\|F\|\leq 1$. For every $n>0$, choose $\theta_n$ from the unit circle such that
\begin{equation}\label{EqAbsValueArg}\big|\sum_{h\leq H}\beta(n+h)F(g(n,h)\Gamma)\big|=\theta_n\sum_{h\leq H}\beta(n+h)F(g(n,h)\Gamma).\end{equation}

Split $(0,H]$ into $W^2$ subintervals $I_1,\cdots, I_k$ of equal lengths $W^{-2}H$. Then for each $n$, the arithmetic progression $[H]$ is decomposed as the disjoint union
$$[H]=\bigsqcup_{\bfj\in\cJ}\cI_{n,\bfj}$$
 of arithmetic progressions $$\cI_{n,\bfj}=\{h\in I_k\cap\bN: n+h\equiv j (\mod q)\},$$ where \begin{equation}\cJ=\{(k,j): 1\leq k\leq W^2, 0\leq j\leq q-1\}.\end{equation}

 Remark that
 \begin{equation}\label{EqEnsembleSize}\#\cJ=W^2q\in(\frac12W^3,W^3].\end{equation} Thus the length of the arithmetic progression $\cI_{n,\bfj}$ satisfies
\begin{equation}\label{EqALength}\#\cI_{n,\bfj}\in [W^{-3}H,2W^{-3}H)\end{equation}

Because $\epsilon$ is $(W,(N,H))$-smooth,  $\di_G(\epsilon(n,h), \id_G)\leq W$ for all $(n,h)\in[N]\times[H]$. Moreover, for any given $1\leq k\leq W^2$, $\di_G(\epsilon(n,h),\epsilon(n,h'))\leq\frac{W}H\cdot W^{-2}H\leq W^{-1}$ for all $h, h'\in I_{n,k}$.

For a given pair $(n,\bfj)=(n,k,j)$, Choose $\epsilon_{n,\bfj}=\epsilon(n,h)$ for the smallest $h\in \cI_{n,\bfj}$. As $\cI_{n,\bfj}\subseteq  I_{n,k}$, we know \begin{equation}\label{EqepsilonSize}d_G(\epsilon_{n,\bfj},\epsilon(n,h))\leq W^{-1}, \forall h\in  \cI_{n,\bfj}.\end{equation} Then \begin{equation}\label{EqepsilonNorm}\di_G(\epsilon_{n,\bfj},\id_G)\ll W.\end{equation}

Choose a rational element $\gamma_{n,\bfj}$ from  such that $\gamma_{n,\bfj}\Gamma=\gamma(n,h)\Gamma$ for any $h\in \cI_{n,\bfj}$. The value of  $\gamma_{n,\bfj}$ can in fact be chosen to be independent of the choice of $h\in\cI_{n,\bfj}$ and $q$-periodic in $n$, because $\cI_{n,\bfj}\subset q\bZ+j$ and $\gamma(n,h)$ is $q$-periodic in both $n$ and $h$. As $\gamma(n,h)$ is $W$-rational, and $\gamma_{n,\bfj}=\gamma(n,h)\xi$  for some $\xi\in\Gamma$, $\gamma_{n,\bfj}$ is $W^{O(1)}$-rational by Lemma \ref{LemRational}. Moreover, we may choose $\gamma_{n,\bfj}$ from the fundamental domain $\psi_\cV^{-1}([0,1)^m)$. In particular, by \cite{GT12a}*{Lemma A.4}, \begin{equation}\label{EqgammaNorm}\di_G(\gamma_{n,\bfj},\id_G)\ll R^{O(1)}.\end{equation}

Define $G_{n,\bfj}$ by $G_{n,\bfj}=\gamma_{n,\bfj}^{-1}G'\gamma_{n,\bfj}$ and $\Gamma_{n,\bfj}=G_{n,\bfj}\cap\Gamma$.

\begin{lemma}\label{LemPeriodicLattice}The following properties are true:
\begin{enumerate}
\item $G_{n,\bfj}$ is a $W^{O(1)}$-rational subgroup and $\Gamma_{n,\bfj}$ is a lattice of it;
\item The assignments $G_{n,\bfj}$ and $\Gamma_{n,\bfj}$ are $q$-periodic in $n$;
\item $G_{n,\bfj}$ has a $W^{O(1)}$ -rational Mal'cev basis $\cV_{n,\bfj}$ adapted to $((G_{n,\bfj})_\bullet, \Gamma_{n,\bfj})$ that consists of $W^{O(1)}$-rational combinations of elements from $\cV$. Here $(G_{n,\bfj})_\bullet$ consists of the subgroups $(G_{n,\bfj})_i=G_{n,\bfj}\cap G_i$.
\end{enumerate}\end{lemma}

\begin{proof}

Because $\gamma_{n,\bfj}$ is $W^{O(1)}$-rational and $G'$ is a $W$-rational subgroup, by \cite{GT12a}*{Lemma A.13}, $G_{n,\bfj}$ is a $W^{O(1)}$-rational subgroup. As $\gamma_{n,\bfj}$ is $q$-periodic in $n$, so are the correspondences from $(n,\bfj)$ to $G_{n,\bfj}$ and $\Gamma_{n,\bfj}$. The last property is given by \cite{GT12a}*{Proposition A.10}.
\end{proof}

Define $g_{n,\bfj}(h)=\gamma_{n,\bfj}^{-1}g'(n,h)\gamma_{n,\bfj}\in G_{n,\bfj}$. Then $g_{n,\bfj}\in\Poly(\bZ,(G_{n,\bfj})_\bullet)$ and \begin{equation}\label{EqMnfdDecomp}\begin{aligned}g(n,h)\Gamma=&\epsilon(n,h)g'(n,h)\gamma(n,h)\Gamma=\epsilon(n,h)g'(n,h)\gamma_{n,\bfj}\Gamma\\
=&\epsilon(n,h)\gamma_{n,\bfj}g_{n,\bfj}(h)\Gamma,\ \forall h\in\cI_{n,\bfj}.\end{aligned}\end{equation}

We then define a new function $F_{n,\bfj}: G_{n,\bfj}/\Gamma_{n,\bfj}\to\bC$ by \begin{equation}F_{n,\bfj}(g\Gamma_{n,\bfj})=\theta_nF(\epsilon_{n,\bfj}\gamma_{n,\bfj}g\Gamma).\end{equation}
Note that $F_{n,\bfj}$ is well-defined because if $g=\hg\eta$ with $\eta\in\Gamma_{n,\bfj}\subset\Gamma$, then $g\Gamma=\hg\Gamma$.

 By  \eqref{EqepsilonNorm}, \eqref{EqgammaNorm} and \cite{GT12a}*{Lemma A.5} and
 \begin{equation}\label{EqSubseqLip}\|F_{n,\bfj}\|_{G_{n,\bfj}/\Gamma_{n,\bfj}}\leq (WR^{O(1)})^{O(1)}\|F\|_{G/\Gamma}\leq W^{O(1)}.\end{equation}

 \begin{lemma}\label{LemGenericED}Suppose $\ref{CONSTImplicit}=O(1)$ is sufficiently large and $\ref{BONSTTotalED}\geq 10\ref{CONSTImplicit}$. There exists a subset $\cN\subseteq[N]$ such that \begin{equation}\label{EqcN}\#\cN\geq (1-W^{-\ref{BONSTTotalED}})N\end{equation} and for all $(n,\bfj)\in\cN\times\cJ$, the sequence $\{g_{n,\bfj}(h)\Gamma_{n,\bfj}\}_{h\in[H]}$ is totally $W^{-\ref{CONSTImplicit}^{-1}\ref{BONSTTotalED}}$-equidistributed in $G_{n,\bfj}/\Gamma_{n,\bfj}$.\end{lemma}

\begin{proof}By property \eqref{ThmFactorization2} in Theorem \ref{ThmFactorization}, it suffices to show that if $\{g_{n,\bfj}(h)\Gamma_{n,\bfj}\}_{h\in[H]}$ is not  totally $W^{-\ref{CONSTImplicit}^{-1}\ref{BONSTTotalED}}$-equidistributed, then $\{g'(n,h)\Gamma'\}_{h\in[H]}$ is not totally $W^{-\ref{BONSTTotalED}}$-equidistributed in $G'/\Gamma'$.

Consider the lattice $\Gamma'_{n,\bfj}=\gamma_{n,\bfj}\Gamma_{n,\bfj}\gamma_{n,\bfj}$ in $G'$. Then $G'/\Gamma'_{n,\bfj}$ is isomorphic to $G_{n,\bfj}/\Gamma_{n,\bfj}$ via the conjugacy $\Ad_{\gamma_{n,\bfj}}$ by $\gamma_{n,\bfj}$. Let $\cV'_{n,\bfj}$ be the image of $\cV_{n,\bfj}$ under $\Ad_{\gamma_{n,\bfj}}$, then it is a Mal'cev basis adapted to $(G'_\bullet,\Gamma'_{n,\bfj})$. Because of the bound \eqref{EqgammaNorm} and \cite{GT12a}*{Lemma A.5}, $\Ad_{\gamma_{n,\bfj}}$ is $R^{O(1)}$-Lipschitz continuous. As $W\geq R$ and $g'(n,h)=\Ad_{\gamma_{n,\bfj}}g_{n,\bfj}(h)$, the sequence $\{g'(n,h)\Gamma'_{n,\bfj}\}_{h\in[H]}$ fails to be totally $W^{-\ref{CONSTImplicit}^{-1}\ref{BONSTTotalED}-O(1)}$-equidistributed in $G_{n,\bfj}/\Gamma'_{n,\bfj}$, with respect to the metric induced by $\cV'_{n,\bfj}$.

Moreover, because $\gamma_{n,\bfj}$ is $W$-rational and satisfies the bound \eqref{EqgammaNorm}, it is a rational element of height bounded by $W^{O(1)}$. Since $\cV_{n,\bfj}$ consists of $W^{O(1)}$-rational combinations of elements of $\cV$, by \cite{GT12a}*{Lemma A.11}, so does $\cV'_{n,\bfj}$. We also know that $\cV'$ consists of $W$-rational combinations of elements from $\cV$. Because they are both Mal'cev basis of $G'$, it follows that $\cV'$ consists of $W^{O(1)}$-rational combinations of elements from $\cV'_{n,\bfj}$. Hence by Corollary \ref{CorChangeLattice},
  the sequence $\{g'(n,h)\Gamma'\}_{h\in[H]}$ fails to be totally $W^{-O(\ref{CONSTImplicit}^{-1}\ref{BONSTTotalED}+O(1))}$-equidistributed in $G_{n,\bfj}/\Gamma'$, with respect to the metric induced by $\cV'$. As it will be assumed that $\ref{BONSTTotalED}\geq 10\ref{CONSTImplicit}$, the lemma follows after updating the value of the constant $\ref{CONSTImplicit}=O(1)$.
\end{proof}

By \eqref{EqMnfdDecomp}, \eqref{EqSubseqLip} and \eqref{EqepsilonSize}, for all $h\in\cI_{n,\bfj}$,
\begin{equation}\di_{G/\Gamma}(\epsilon_{n,\bfj}\gamma_{n,\bfj}g_{n,\bfj}(h)\Gamma,g(n,h)\Gamma)\leq W^{-1},\end{equation} and \begin{equation}\label{EqSmoothApprox0}|F_{n,\bfj}(g_{n,\bfj}(h)\Gamma_{n,\bfj})-\theta_n F(g(n,h)\Gamma)|\leq W^{-1}\|F\|.\end{equation}

\begin{lemma}\label{LemSmoothApprox} For all Lipschitz function $F$ on $G/\Gamma$, the sum
\begin{equation}\label{EqSmoothApprox1}\sum_{n\leq N}\Big|\sum_{h\leq H} \beta(n+h)F(g(n,h)\Gamma)\Big|\end{equation}
 is approximated by
\begin{equation}\label{EqSmoothApprox2}\sum_{n\leq N}\sum_{\bfj\in\cJ}\sum_{h\in\cI_{n,\bfj}}
\beta(n+h)F_{n,\bfj}(g_{n,\bfj}(h)\Gamma_{n,\bfj}) ,\end{equation} up to an error bounded by $W^{-1}HN$.
\end{lemma}
\begin{proof}As $[H]=\bigsqcup_{\bfj\in\cJ}\cI_{n,\bfj}$, the claim follows from \eqref{EqAbsValueArg} and \eqref{EqSmoothApprox0}.\end{proof}

For each triple $(n,\bfj)$, decompose $F_{n,\bfj}$ as $\tilde F_{n,\bfj}+ E_{n,\bfj}$ where $E_{n,\bfj}=\int_{G_{n,\bfj}/\Gamma_{n,\bfj}}F_{n,\bfj}$ is a constant and $\tF_{n,\bfj}$ has zero average on $G_{n,\bfj}/\Gamma_{n,\bfj}$. Then \eqref{EqSmoothApprox2} splits into the sum of a major arc part
\begin{equation}\label{EqMajor}\sum_{n\leq N} \sum_{\bfj\in\cJ}\sum_{h\in\cI_{n,\bfj}}E_{n,\bfj}\beta(n+h) .\end{equation}
and a minor art part
\begin{equation}\label{EqMinor}\sum_{n\leq N} \sum_{\bfj\in\cJ}\sum_{h\in\cI_{n,\bfj}}\beta(n+h)\tF_{n,\bfj}(g_{n,\bfj}(h)\Gamma_{n,\bfj}) ,\end{equation}
Note that, \begin{equation}\label{EqMajorTermBd}|E_{n,\bfj}|\leq 1,\end{equation}
\begin{equation}\label{EqMinorTermBd}\|\tF_{n,\bfj}\|_{G_{n,\bfj}/\Gamma_{n,\bfj}}\leq2\|F_{n,\bfj}\|_{G_{n,\bfj}/
\Gamma_{n,\bfj}}\ll W^{O(1)}.\end{equation}
\begin{equation}\label{EqMinorTermC0Bd}\|\tF_{n,\bfj}\|_{C^0(G_{n,\bfj}/\Gamma_{n,\bfj})}\leq 2.\end{equation}

\section{Major arc estimate}\label{SecMajor}

The major arc estimate will concern only multiplicative functions $\beta$ that are non-pretentious as defined by Granville and Soundararajan \cite{GS07}. Given two $1$-bounded multiplicative functions $\beta, \beta'$ and a parameter
$X\ge 1$, a distance $\bD(\beta, \beta';X) \in  [0,+\infty)$ is defined by the formula
$$\bD(\beta,\beta';X) := \left( \sum_{p\le X} \frac{1-\re(\beta(p)\overline{\beta'(p)})}{p} \right)^{1/2}.
$$
It is known that this gives a (pseudo-)metric on $1$-bounded multiplicative functions; see \cite[Lemma 3.1]{GS07}. Moreover, let
\begin{equation}\label{DefM}M(\beta; X) := \inf_{|t|\le X} \bD(\beta, n\mapsto n^{it}; X)^2\end{equation}
and
\begin{equation}\begin{aligned}\label{DefM2}
M(\beta;X,Y):&= \inf_{q\le Y; \chi \, (q)} M(\beta\overline{\chi};X)\\
&=\inf_{|t|\le X; q\le Y; \chi\, (q)} \bD(\beta, n\mapsto \chi(n) n^{it};X)^2,
\end{aligned}\end{equation}
where $\chi$ ranges over all Dirichlet characters of modulus $q\le Y$.

In addition, we also define \begin{equation}\label{EqDefM1}\tM(\beta,X,Y)=\displaystyle\inf_{X'\geq X}M(\beta,X',Y).\end{equation} Remark that $\tM$ is increasing in $X$ and decreasing in $Y$.

 Instead of \eqref{EqMajor}, we will first estimate
\begin{equation}\label{EqMajorDense}\sum_{n\leq N}\sum_{\bfj\in\cJ}\sum_{h\in\cI_{n,\bfj}}E_{n,\bfj}1_\cS\beta(n+h) .\end{equation}
In this part, we will prove

\begin{proposition}\label{PropMajor}
Assuming Hypothesis \ref{HypMain}, Notation \ref{NotationBilinear} and the following inequalities:
\begin{equation}\label{EqPropMajorCond} \frac{\log\log H}{\log H}<\epsilon<\frac 1{500};\ 10\leq R_0\leq R\leq H^{\frac\epsilon{\ref{CONSTWRange}\ref{BONSTTotalED}^m}}; \log H<(\log N)^{\frac12}.\end{equation}

Then for all $1$-bounded multiplicative function $\beta:\bN\to\bC$ and function $F:G/\Gamma\to\bC$ with $\|F\|\leq 1$, there exists a subset $\cS\subseteq[0,N]\cap\bN$ with $N-\#\cS\ll\epsilon N$, such that
\begin{equation}\label{EqPropMajor}\begin{aligned}&\Big|\sum_{n\leq N}\sum_{\bfj\in\cJ}\sum_{h\in\cI_{n,\bfj}}E_{n,\bfj}1_\cS\beta(n+h) \Big|\\
\ll & \Big(W^{-\frac14}+ W^2e^{-\frac12\tM(\beta,\frac N{W^5}, W )} \tM(\beta,\frac N{W^5}, W )^{\frac12}+ W^2(\log\frac N{W^5})^{-\frac1{100}}\Big)HN.\end{aligned}\end{equation}

Moreover, the choice of $\cS$ only depends on $H$, $N$, and $\epsilon$.
 \end{proposition}

This will result from the following more precise statement.

\begin{proposition}\label{PropMajorPR}
Assume the settings of Theorem \ref{ThmFactorization}, and inequalities \begin{equation}\label{EqPRcondition} 10\leq P_1<Q_1\leq\exp\big((\log N)^{\frac12}\big),\
(\log Q_1)^{480}<P_1; \end{equation}
\begin{equation}\label{EqMajorDensePara}W^{96}\leq P_1<Q_1\leq W^{-4}H.\end{equation}

Then there exists a subset $\cS\subseteq[0,N]\cap\bN$ with \begin{equation}\label{EqMajorSDense}N-\#\cS\ll\frac{\log P_1}{\log Q_1}N,\end{equation} such that for all $1$-bounded multiplicative function $\beta:\bN\to\bC$ and function $F:G/\Gamma\to\bC$ with $\|F\|\leq 1$,
\begin{equation}\label{EqMajorPR1}\begin{aligned}&\Big|\sum_{n\leq N}\sum_{\bfj\in\cJ}\sum_{h\in\cI_{n,\bfj}}E_{n,\bfj}1_\cS\beta(n+h) \Big|\\
\ll & \Big(W^{-\frac14}+ W^2e^{-\frac12\tM(\beta,\frac N{W^5}, W )} \tM(\beta,\frac N{W^5}, W )^{\frac12}+ W^2(\log\frac N{W^5})^{-\frac1{100}}\\
&\hskip1cm + \frac{(\log H)^{\frac16}}{P_1^{\frac1{48}}}\Big)HN.\end{aligned}\end{equation}

Moreover, the choice of $\cS$ only depends on $H$, $N$, $P_1$ and $Q_1$.
 \end{proposition}

 \begin{proof}[Proof of Proposition \ref{PropMajor} assuming Proposition \ref{PropMajorPR}] Let $Q_1=H^{\frac{96}{100}}$ and $P_1=Q_1^{500\epsilon}$.
The inequalities in \eqref{EqPropMajorCond}, together with the fact that $W\in[R,R^{\ref{CONSTWRange}\ref{BONSTTotalED}^m}]$, imply $W <H^\epsilon<H^{\frac1{500}}$, $Q_1<W^{-4}H$, and $P_1=H^{480\epsilon}$, which in turn guarantee \eqref{EqPRcondition} and \eqref{EqMajorDensePara}.

We also have $$\frac{(\log H)^{\frac16}}{P_1^{\frac1{48}}}\leq\frac{(\log H)^{\frac16}}{H^{2\epsilon}}<H^{(-2+\frac16)\epsilon}<H^{-\epsilon}<W^{-1},$$ and $$\frac{\log P_1}{\log Q_1}=500\epsilon\ll\epsilon.$$  So Proposition \ref{PropMajor} follows from \eqref{EqMajorPR1}. Notice that $\cS$ only depends on $N$, $H$, $P_1$ and $Q_1$, where as $P_1$ and $Q_1$ are determined by $H$ and $\epsilon$.
\end{proof}

The following constants are defined in \cite{MRT16}*{\S2}:
\begin{definition}\label{DefDenseS}Given $P_1$, $Q_1$ as in \eqref{EqPRcondition}, let $P_r$, $Q_r$ be inductively defined by $$P_r=\exp(r^{4r}(\log Q_1)^{r-1}\log P_1),\ Q_r=\exp(r^{4r+2}(\log Q_1)^r).$$ Let $r_+$ be the largest index such that $Q_{r_+}\leq\exp\big(\frac{(\log N)^{\frac12}}2\big)$. Also define
$$\cS_{P_1,Q_1,N}=\{n\leq N: n\text{ has at least one prime factor in }[P_r,Q_r], \forall 1\leq r\leq r_+\}.$$
\end{definition}

\begin{lemma}\label{LemDenseS}\cite{MRT16}*{Lemma 2.2} $\#(\{1\leq n\leq N\}\setminus\cS_{P_1,Q_1,N})\ll\frac{\log P_1}{\log Q_1}N.$\end{lemma}

In addition to the conditions in Definition \ref{DefDenseS}, we shall also assume $H\ll N$ and \eqref{EqMajorDensePara}.
We will also write simply \begin{equation}\cS=\cS_{P_1,Q_1,N}\end{equation} when it does not cause ambiguity. Clearly, the construction of $\cS$ depends only on $N$,  $P_1$ and $Q_1$.

Following \cite{MRT16}*{p2177-2178}, denote by $\hat\beta $ the completely multiplicative function determined by $\hat\beta (p)=\beta(p)$ for all prime numbers $p$. Then the Dirichlet inverse of $\hat\beta $ is $\mu\hat\beta $, and thus $\beta=\hat\beta *\eta$, where $\eta=\beta*\mu\hat\beta $ is the Dirichlet convolution between $\beta$ and $\mu\hat\beta$. Then the function $\eta$ is multiplicative,  bounded by $2$ in absolute value, and satisfies
\begin{equation}\label{EqInversionBound}\sum_{n=1}^\infty |\eta(n)|n^{-(\frac12+\sigma)}=O_\epsilon(1)\end{equation}
for all $\sigma>0$.  Note that $\bD(\beta, \beta';N)=\bD(\hat\beta, \beta';N)$ for all $\beta'$.

For $1\leq k\leq W^2$  let $$f_{n,k}(h)=\sum_{j=0}^{q-1} E_{n,(k,j)}1_{\cI_{n,(k,j)}}(h)$$ on $I_{n,k}$. Then $f_{n,k}$ is bounded by $1$ in absolute value and $q$-periodic on $I_k\cap\bN$. Furthermore,
\begin{equation}\label{EqMajor1}\begin{aligned}\eqref{EqMajorDense}=&\sum_{n\leq N}\sum_{k\leq W^2}\sum_{h\in I_k\cap\bN}1_\cS\beta(n+h)f_{n,k}(h) \\
=&\sum_{n\leq N}\sum_{k\leq W^2}\sum_{a\in\bN}\eta(a)\sum_{\substack{b\in\bN\\ab\in n+I_k }}1_\cS(ab)\hat\beta(b)f_{n,k}(ab-n)
\end{aligned}\end{equation}

By \eqref{EqInversionBound}, the contribution of terms with $a>W$ is bounded:

\begin{lemma}\label{LemMajorTail}$\displaystyle\bigg|\sum_{n\leq N}\sum_{k\leq W^2}\sum_{a>W}\eta(a)\sum_{\substack{b\in\bN\\ab\in n+I_k }}1_\cS(ab)\hat\beta(b)f_{n,k}(ab-n) \bigg|\ll W^{-\frac14}HN$.\end{lemma}
\begin{proof} For every $x\in[0,N]$ and $k\leq W^2$,
\begin{equation}\label{EqLemMajorTail1}\begin{aligned}&\bigg|\sum_{a>W}\eta(a)\sum_{\substack{b\in\bN\\ab\in n+I_k }}1_\cS(ab)\hat\beta(b)f_{n,k}(ab-n)\bigg|\\
\leq&\sum_{a>W}|\eta(a)|\cdot 2a^{-1}W^{-2}H\leq \sum_{a>W}|\eta(a)|a^{-\frac34}\cdot 2W^{-\frac14}\cdot W^{-2}H\\
\ll&  W^{-\frac14}\cdot W^{-2}H.
\end{aligned}\end{equation}
The lemma follows by summing over $1\leq k\leq W^2$ and $n\leq N$.
\end{proof}

Next, we aim to bound \begin{equation}\label{EqMajorTrunk}\begin{aligned}&\sum_{n\leq N}\sum_{k\leq W^2}\sum_{a\leq W}\eta(a)\sum_{\substack{b\in\bN\\ab\in n+I_k }}1_\cS(ab)\hat\beta(b)f_{n,k}(ab-n) \\
=&\sum_{n\leq N}\sum_{k\leq W^2}\sum_{a\leq W}\eta(a)\sum_{\substack{b\in\bN\\ab\in n+I_k }}1_\cS\hat\beta(b)f_{n,k}(ab-n) .
\end{aligned}\end{equation} Here the equality is because of the fact that, as $a\leq W<P_1<Q_1$, $b\in\cS$ if and only if $ab\in\cS$.

Given $a\leq W$, $k\leq W^2<P_1$ and $n\leq N$, decompose $\{b\in\bN:ab\in n+I_k \}$ according to $u=\gcd(b,q)$:
\begin{equation}\label{EqLemDirichlet3}\begin{aligned}&\sum_{\substack{b\in\bN \\ ab\in n+I_k }}1_\cS\hat\beta(b)f_{n,k}(ab-n)\\
=&\sum_{u|q}\sum_{\substack{ab\in n+I_k \\(b,q)=u}}1_\cS\hat\beta(b)f_{n,k}(ab-n)\\
=&\sum_{u|q}\hat\beta(u)\sum_{\substack{auv\in n+ I_k\\(v,\frac{q}u)=1}}1_\cS\hat\beta(v)f_{n,k}(auv-n),
\end{aligned}\end{equation}
where the last equality uses the fact that $1_\cS(uv)\hat\beta(uv)=1_\cS(v)\hat\beta(u)\hat\beta(v),$ which is because $\hat\beta$ is completely multiplicative and $u\leq q\leq W< P_1$.

The Dirichlet characters of conductor $\frac{q}u$ form an orthonormal basis of the $l^2$-space on the finite abelian group $\big(\bZ/(\frac{q}u)\bZ\big)^\times$ .

 Since $f_{n,k,a,u}:v\to f_{n,k}(auv-n) 1_{(v,\frac{q}u)=1}$ is $\frac{q}u$-periodic, it can be decomposed as a linear combination$\displaystyle\sum_{\chi\ \mathrm{mod}^*\  \frac{q}u}w_{n,k,a,u,\chi}\chi$ of such characters. Then, \begin{equation}\label{EqLemDirichlet4}\sum_{\chi\ \mathrm{mod}^*\  \frac{q}u}|w_{n,k,a,u,\chi}|^2\leq \|f_{n,k,a,u}\|_{l^\infty}\leq 1.\end{equation}

It follows from this and \eqref{EqLemDirichlet3} that, by Cauchy-Schwarz inequality,
\begin{equation}\begin{aligned}&\bigg|\sum_{\substack{b\in\bN \\ ab\in n+I_k }}1_\cS\hat\beta(b)f_{n,k}(ab-n)\bigg|^2\\
=&\bigg|\sum_{u|q}\hat\beta(u)\sum_{\chi\ \mathrm{mod}^*\frac{q}u}w_{n,k,a,u,\chi}\sum_{\substack{v\in\bN\\auv\in n+ I_k}}1_\cS\hat\beta(v)\chi(v)\bigg|^2\\
\leq&\big(\sum_{u|q}|\hat\beta(u)|^2\big) \cdot\bigg(\sum_{u|q}\bigg|\sum_{\chi\ \mathrm{mod}^*\frac{q}u}w_{n,k,a,u,\chi}\sum_{v\in(\frac n{au}+\frac1{au}I_k)\cap\bN}1_\cS\hat\beta(v)\chi(v)\bigg|^2\bigg)\\
\leq &q\bigg(\sum_{u|q}\Big(\sum_{\chi\ \mathrm{mod}^* \frac{q}u}|w_{n,k,a,u,\chi}|^2\Big)\Big(\sum_{\chi\ \mathrm{mod}^*\frac{q}u}\Big|\sum_{v\in(\frac n{au}+\frac1{au}I_k)\cap\bN}1_\cS\hat\beta(v)\chi(v)\Big|^2\bigg)\bigg)\\
\leq &q\bigg( \sum_{\substack{u|q\\ \chi\ \mathrm{mod}^*\ \frac{q}u}}\Big|\sum_{v\in(\frac n{au}+\frac1{au}I_k)\cap\bN}1_\cS\hat\beta(v)\chi(v)\Big|^2\bigg).\\\end{aligned}\end{equation}

Therefore, again by Cauchy-Schwarz inequality,
\begin{equation}\label{EqLemDirichlet4}\begin{aligned}
&\bigg|\sum_{n\leq N}\sum_{\substack{b\in\bN\\ab\in n+I_k }}1_\cS\hat\beta(b)f_{n,k}(ab-n) \bigg|^2\\
\leq & N \sum_{n\leq N}\Big|\sum_{\substack{b\in\bN \\ ab\in n+I_k }}1_\cS\hat\beta(b)f_{n,k}(ab-n)\Big|^2 \\
\leq &  N \sum_{n\leq N} q \sum_{\substack{u|q\\ \chi\ \mathrm{mod}^*\ \frac{q}u}}\Big|\sum_{v\in(\frac n{au}+\frac1{au}I_k)\cap\bN}1_\cS\hat\beta(v)\chi(v)\Big|^2 \\
\leq &  W  N \sum_{n\leq N} \sum_{\substack{u\leq  W  \\\mathrm{cond}\chi\leq\frac{ W }u}}\Big|\sum_{v\in(\frac n{au}+\frac1{au}I_k)\cap\bN}1_\cS\hat\beta(v)\chi(v)\Big|^2 \\
\leq & W  N \sum_{\substack{u\leq  W \\\mathrm{cond}\chi\leq\frac{ W }u}}au\sum_{n\leq \frac N{au}}\Big|\sum_{v\in(n+\frac1{au}I_k)\cap\bN}1_{\cS_{P_1,Q_1,\frac{N+H}{au}}}\hat\beta(v)\chi(v)\Big|^2.
\end{aligned}\end{equation}

The sum within \eqref{EqLemDirichlet4} is controlled by the work of Matom\"aki-Radziwi\l{}\l{}-Tao on averages of multiplicative functions on short intervals.

\begin{theorem}\label{ThmMRT}(Matom\"aki-Radziwi\l{}\l{}-Tao) \cite{MRT16}*{Thm A.2} Suppose that $10<P_1<Q_1<H$ and $\
(\log Q_1)^{480}<P_1$, then for all sufficiently large $N$,  $1$-bounded multiplicative function $\beta$ and Dirichlet character of modulus bounded by $Y$, $$\begin{aligned}&\sum_{N<n\leq 2N}\Big|\sum_{n\leq v\leq n+H_0}1_{\cS_{P_1,Q_1,2N+H_0}} \beta(v)\chi(v)\Big|^2 \\
\ll&\Big(e^{- M(\beta,N,Y)} M(\beta,N,Y)+\frac{(\log H_0)^{\frac13}}{P_1^{\frac1{12}}}+(\log N)^{-\frac1{50}}\Big)H_0^2N,\end{aligned}$$
where $M(\beta,N,Y)$ is defined by \eqref{DefM2}.\end{theorem}

\begin{corollary}\label{CorMRT}Assuming the conditions \eqref{EqPRcondition} and \eqref{EqMajorDensePara}, for all positive integers $k\leq  W^2, T\leq  W^2$, $1$-bounded multiplicative function $\beta$, and primitive characters $\chi$ of conductor bounded by $ W $,
$$\begin{aligned}&T\sum_{n\leq \frac NT}\Big|\sum_{v\in (n+\frac1TI_k)\cap\bN}1_{\cS_{P_1,Q_1,\frac{N+H}T}}\hat\beta(v)\chi(v)\Big|^2 \\
\ll&\Big(W^{-7}+e^{-\tM(\beta,\frac N{W^5}, W )}\tM(\beta,\frac N{W^5}, W )+\frac{(\log H)^{\frac13}}{P_1^{\frac1{12}}}+\big(\log\frac N{W^5}\big)^{-\frac1{50}}\Big)\frac{H^2N}{T^2}.\end{aligned}$$  \end{corollary}

\begin{proof}Decompose $[0,\frac NT]$ into dyadic intervals $[\frac N{2^iT},\frac N{2^{i-1}T}]$ for $i=1,\cdots, \lceil 3 \log_2 W\rceil$, and $[0,\frac N{2^{\lceil 3\log_2 W\rceil}T}]$.

The contribution of the last interval can be bound trivially by $$T\cdot\frac N{ W^3T}\cdot (\frac H{ W^2 T})^2\ll W^{-7}\frac{H^2N }{T^2}.$$

By Theorem \ref{ThmMRT}, with $H_0=\frac H{W^2T}\leq W^{-2}H$, the contribution from the dyadic intervals is
$$\begin{aligned}
\ll& \sum_{i\leq \lceil3\log_2 W\rceil}\Big(e^{-M(\hat\beta,\frac N{2^iT}, W )}M(\hat\beta,\frac N{2^iT}, W )+\frac{(\log H)^{\frac13}}{P_1^{\frac1{12}}}+(\log\frac N{2^iT})^{-\frac1{50}}\Big)\frac{H^2N }{2^{2i}T^2}\\
\ll& \Big(e^{-\tM(\hat\beta,\frac N{W^5}, W )}\tM(\hat\beta,\frac N{W^5}, W )+\frac{(\log H)^{\frac13}}{P_1^{\frac1{12}}}+(\log\frac N{W^5})^{-\frac1{50}}\Big)\frac{H^2N }{T^2}.
\end{aligned}$$
The corollary follows because $\tM(\beta,\cdot,\cdot)$ and $\tM(\hat\beta,\cdot,\cdot)$  have the same value.
\end{proof}

Therefore, with $\square$ denoting the bracketed coefficient in Corollary \ref{CorMRT},
\begin{equation}\eqref{EqLemDirichlet4}
\ll  W  N \sum_{u\leq  W }\frac{ W }u\cdot \square\frac{H^2N }{(au)^2}\ll\square\frac{ W^2 H^2N^2}{a^2}.\end{equation}
In other words,
\begin{equation}\label{EqLemDirichlet5}\Bigg|\sum_{n\leq N}\sum_{\substack{b\in\bN \\ ab\in n+I_k }}1_\cS\hat\beta(b)f_{n,k}(ab-n) \Bigg|\ll a^{-1}\square^\frac12 W HN\end{equation} for all $a\leq W$, $k\leq  W^2$.

\begin{lemma}\label{LemMajorTrunk} Assuming the conditions \eqref{EqPRcondition} and \eqref{EqMajorDensePara}, we have
$$\bigg|\sum_{n\leq N}\sum_{k\leq W^2}\sum_{a\leq W}\eta(a)\sum_{\substack{b\in\bN \\ ab\in n+I_k }}1_\cS\hat\beta(b)f_{n,k}(ab-n) \bigg|\\
\ll  \square^\frac12 W^3HN.$$
\end{lemma}
\begin{proof}Summing \eqref{EqLemDirichlet5} over $k$ and $a$, one can see that the left hand side is bounded by $$\sum_{a\leq W}\eta(a)a^{-1}\square^\frac12 W^3HN.$$ which is in turn by \eqref{EqInversionBound} bounded by the right hand side up to a multiplicative constant.\end{proof}

\begin{proof}[Proof of Proposition \ref{PropMajorPR}] By merging Lemmas \ref{LemMajorTail}, Lemma \ref{LemMajorTrunk} into \eqref{EqMajor1}, we see that
$$\begin{aligned}&|\eqref{EqMajorDense}|\\
\ll&W^{-\frac14}H N + W^2\Big(W^{-5}+e^{-\tM(\beta,\frac N{W^5}, W )} \tM(\beta,\frac N{W^5}, W )+\frac{(\log H)^{\frac13}}{P_1^{\frac1{12}}}\\
&\hskip1cm+(\log\frac N{W^5})^{-\frac1{50}}\Big)^\frac12HN\\
\ll&\Big(W^{-\frac14}+ W^2e^{-\frac12\tM(\beta,\frac N{W^5}, W )} \tM(\beta,\frac N{W^5}, W )^{\frac12}+ W^2(\log\frac N{W^5})^{-\frac1{100}}\\
&\hskip1cm+ W^2\frac{(\log H)^{\frac16}}{P_1^{\frac1{24}}}\Big)HN,\end{aligned}$$ which is in turn bounded by the right hand side up to a constant multiple.

The proposition follows, thanks to Lemma \ref{LemDenseS} and the fact that $W^2\leq P_1^{\frac1{48}}$. \end{proof}

\section{Minor arc estimate}\label{SecMinorI}

In Sections \ref{SecMinorI} and \ref{SecMinorII}, we will provide a bound to $\eqref{EqMinor}$  under appropriate hypothesis.

\begin{proposition}\label{PropMinor}
Assuming Hypothesis \ref{HypMain} and Notation \ref{NotationBilinear}, the constant $\ref{CONSTImplicit}$ being sufficiently large, and the following inequalities:
\begin{equation}\label{EqPropMinorCond} 0<\epsilon<\frac 1{100}; \ref{BONSTTotalED}\geq \ref{CONSTImplicit}; 10\leq R_0\leq R\leq H^{\frac\epsilon{\ref{CONSTWRange}\ref{BONSTTotalED}^{m+1}}}.\end{equation}

Then for all $1$-bounded multiplicative function $\beta:\bN\to\bC$ and function $F:G/\Gamma\to\bC$ with $\|F\|\leq 1$, there exists a subset $\cS\subseteq[0,N]\cap\bN$ with $N-\#\cS\ll\epsilon N$, such that
\begin{equation}\label{EqPropMajor}\begin{aligned}&\Big|\sum_{n\leq N}\sum_{\bfj\in\cJ}\sum_{h\in\cI_{n,\bfj}}1_\cS\beta(n+h)\tF_{n,\bfj}(g_{n,\bfj}(h)\Gamma_{n,\bfj}) \Big|\\
\ll & (W^{-\ref{CONSTImplicit}^{-1}\ref{BONSTTotalED}}\log H+H^{-\epsilon})HN.\end{aligned}\end{equation}

Moreover, the choice of $\cS$ only depends on $H$, $N$, and $\epsilon$.
 \end{proposition}

Following \cite{MRT16}*{\S3}, let $\cP$ be the set of primes in $[P_1,Q_1]$ for some fixed values $W<P_1<Q_1<H$. A priori, $P_1$, $Q_1$ do not have to assume the same values as in \S\ref{SecMajor}.

\begin{lemma}\label{LemSqFree}Under the assumptions of Proposition \ref{PropMinor}, there exists a subset $\cS\subseteq[0,N]\cap\bN$ with $N-\#\cS\ll\frac{\log P_1}{\log Q_1}N$, such that
for all $n\leq N$,
$$\sum_{\substack{h\leq H\\n+h\in\cS}}\bigg|\beta(n+h)-\sum_{p\in\cP}\sum_{l\in\bN}\frac{1_{pl=n+h}\beta(p)\beta(l)}{1+\#\{q\in\cP: q|l\}}\bigg|\ll\frac H{P_1}.$$ The construction of $\cS$ only depends on $N$ and $P_1$, $Q_1$. \end{lemma}
\begin{proof} Define $$\cS=\{n\leq N: \exists p\in \cP, p|n\}$$ and $$\cF=\{n\in\bN\leq N: p^2\nmid n, \forall p\in\cP\}.$$ Note that these definitions depends only on $N$, $P_1$ and $Q_1$.

By Lemma \ref{LemDenseS}, $N-\#\cS\ll\frac{\log P_1}{\log Q_1}N$.

Decompose the sum on the left hand side as $\displaystyle \sum_{\substack{h\leq H\\n+h\in S\setminus\cF}}+\sum_{\substack{h\leq H\\n+h\in\cS\cap\cF}}$. We will bound the two components separately.

Remark first that, when $n+h\in\cS$, \begin{equation}\begin{aligned}\sum_{p\in\cP}\sum_{l\in\bN}\frac{1_{pl=n+h}}{1+\#\{q\in\cP: q|l\}}=&\sum_{p\in\cP}\sum_{l\in\bN}\frac{1_{pl=n+h}}{1_{p^2|n+h}+\#\{q\in\cP: q|n\}}\\\leq &\sum_{p\in\cP}\frac{1_{p|n+h}}{\#\{q\in\cP: q|n\}}=1.\end{aligned}\end{equation} In particular, the equality holds when $n\in\cS\cap\cF$.

If $n+h\in\cS\cap\cF$, then for all $p\in\cP$ and $l\in\bN$ such that $pl=n+h$, $p\nmid l$ and thus $\beta(n+h)=\beta(p)\beta(l)$. Hence
$$\bigg|\beta(n+h)-\sum_{p\in\cP}\sum_{l\in\bN}\frac{1_{pl=n+h}\beta(p)\beta(l)}{1+\#\{q\in\cP: q|l\}}\bigg|=\bigg|\beta(n+h)-\sum_{p\in\cP}\sum_{l\in\bN}\frac{1_{pl=n+h}\beta(n+h)}{1+\#\{q\in\cP: q|l\}}\bigg|=0.$$ So
\begin{equation}\label{EqLemSqFree2}\sum_{\substack{h\leq H\\n+h\in\cS\cap\cF}}=0\end{equation}

On the other hand, if $n+h\in\cS\setminus\cF$, then $$\bigg|\beta(n+h)-\sum_{p\in\cP}\sum_{l\in\bN}\frac{1_{pl=n+h}\beta(p)\beta(l)}{1+\#\{q\in\cP: q|l\}}\bigg|\leq 1+\sum_{p\in\cP}\sum_{l\in\bN}\frac{1_{pl=n+h}}{1+\#\{q\in\cP: q|l\}}\leq 2.$$
So \begin{equation}\label{EqLemSqFree3}\sum_{\substack{h\leq H\\n+h\in\cS\setminus\cF}}
\leq2\sum_{\substack{h\leq H\\n+h\in\cS\setminus\cF}}1
\leq2\sum_{h\leq H}\sum_{p\in\cP}1_{p^2|n+h}\leq2\sum_{p\geq P_1}\frac H{p^2}\ll\frac H{P_1}.\end{equation}
It now suffices to add together \eqref{EqLemSqFree2} and \eqref{EqLemSqFree3}.\end{proof}

\begin{corollary}\label{CorSqFree}The  integral
\begin{equation}\label{EqMinorS}\sum_{n\leq N} \sum_{\bfj\in\cJ}\sum_{h\in\cI_{n,\bfj}}1_\cS\beta(n+h)\tF_{n,\bfj}(g_{n,\bfj}(h)\Gamma_{n,\bfj}) ,\end{equation}is approximated by
\begin{equation}\label{EqCorSqFree}\sum_{n\in\cN}\sum_{\bfj\in\cJ}\sum_{p\in\cP}\sum_{l\in\bN}\frac{1_{pl\in n+\cI_{n,\bfj}}\beta(p)\beta(l)}{1+\#\{q\in\cP: q|l\}}\tF_{n,\bfj}(g_{n,\bfj}(pl)\Gamma_{n,\bfj}) \end{equation}
within an error of $O(P_1^{-1}+W^{-\ref{BONSTTotalED}})\cdot H N $.\end{corollary}

Here the set $\cN\subseteq[N]$ is chosen as in \eqref{EqcN}.
\begin{proof}The corollary directly follows from the lemma above and the inequality  \eqref{EqcN}.\end{proof}

Take $P_1=2^{s_-}$ and $Q_1=2^{s_+}$ for integers $s_-<s_+$. The expression \eqref{EqCorSqFree} splits into the sum \begin{equation}\label{EqMinorSDya}\sum_{s\in(s_-,s_+]}\sum_{n\in\cN}\sum_{\bfj\in\cJ}\sum_{ p\in(2^{s-1},2^s]}\sum_{l\in\bN}\frac{1_{pl\in n+\cI_{n,\bfj}}\beta(p)\beta(l)}{1+\#\{q\in\cP: q|l\}}\tF_{n,\bfj}(g_{\bfj}(n)\Gamma_{n,\bfj}) ,\end{equation} over all integers $s\in [s_-,s_+]$.

\begin{notation} Here and below, the letter $p$, as well as $p_1$, $p_2$, will always refer to prime numbers only.\end{notation}

Observe that, for all given $s$,
\begin{equation}\label{EqMinorS1}\begin{aligned}
&\bigg|\sum_{n\in\cN}\sum_{\bfj\in\cJ}\sum_{ p\in(2^{s-1},2^s]}\sum_{l\in\bN}\frac{1_{pl\in n+\cI_{n,\bfj}}\beta(p)\beta(l)}{1+\#\{q\in\cP: q|l\}}\tF_{n,\bfj}(g_{n,\bfj}(pl)\Gamma_{n,\bfj})\bigg|\\
\leq&\sum_{l\in\bN}\frac{|\beta(l)|}{1+\#\{q\in\cP: q|l\}}\bigg|\sum_{n\in\cN}\sum_{\bfj\in\cJ}\sum_{\substack{ p\in(2^{s-1},2^s]\\pl\in n+\cI_{n,\bfj}}}\beta(p)\tF_{n,\bfj}(g_{n,\bfj}(pl)\Gamma_{n,\bfj})\bigg|\\
\leq&\sum_{l\leq\frac{N+H}{2^{s-1}}}\bigg|\sum_{n\in\cN}\sum_{\bfj\in\cJ}\sum_{\substack{ p\in(2^{s-1},2^s]\\pl\in n+\cI_{n,\bfj}}}\beta(p)\tF_{n,\bfj}(g_{n,\bfj}(pl)\Gamma_{n,\bfj})\bigg|\\
\ll&2^{-\frac s2}N^\frac12\bigg(\sum_{l\leq \frac{N+H}{2^{s-1}}}\bigg|\sum_{n\in\cN}\sum_{\bfj\in\cJ}\sum_{\substack{ p\in(2^{s-1},2^s]\\pl\in n+\cI_{n,\bfj}}}\beta(p)\tF_{n,\bfj}(g_{n,\bfj}(pl)\Gamma_{n,\bfj})\bigg|^2\bigg)^\frac12.
\end{aligned}\end{equation}
This is  because if $\bfj=(k,j)$ and $pl\in n+\cI_{n,\bfj}$, then $2^{s-1}l\leq pl\leq N+H$.

For a configuration $\bfn=(n,\bfj)=(n,k,j)\in\cN\times\cJ$, define an arithmetic progression
\begin{equation}\label{EqIntersection1}\cA_{\bfn,p}=\{l\in\bN: pl\in n+\cI_{n,\bfj}\}=\{l\in\bN: pl-n\in I_{k}, pl\equiv {j} (\mod q)\}\end{equation}

For two such given configurations $$\bfn_1=(n_1,\bfj_1)=(n_1, k_1, j_1),\bfn_2=(n_2,\bfj_2)=(n_2,k_2,j_2)\in\cN\times\cJ,$$ write \begin{equation}\label{EqIntersection}\begin{aligned}\cA_{\bfn_1,\bfn_2,p_1,p_2}=\cA_{\bfn_1,p_1}\cap\cA_{\bfn_2,p_2}.\end{aligned}\end{equation} Then
\begin{equation}\label{EqMinorS2}\begin{aligned}
&\sum_{l\leq \frac{N+H}{2^{s-1}}}\bigg|\sum_{n\in\cN}\sum_{\bfj\in\cJ}\sum_{\substack{ p\in(2^{s-1},2^s]\\pl\in n+\cI_{n,\bfj}}}\beta(p)\tF_{n,\bfj}(g_{n,\bfj}(pl)\Gamma_{n,\bfj}) \bigg|^2\\
=&\sum_{\bfn_1,\bfn_2\in\cN\times\cJ}\sum_{p_1,p_2\in(2^{s-1},2^s]}\sum_{l\in\cA_{\bfn_1,\bfn_2,p_1,p_2}}\beta(p_1)\overline{\beta(p_2)}\\
&\hskip2cm \tF_{\bfn_1}(g_{\bfn_1}(p_1l)\Gamma_{\bfn_1})\overline{\tF_{\bfn_2}(g_{\bfn_2}(p_2l)\Gamma_{\bfn_2})}
\end{aligned}\end{equation}

It will be useful to have an upper bound on the size of $\cA_{\bfn_1,\bfn_2,p_1,p_2}$.

\begin{lemma}\label{LemIntersectionSize}If $p_1>W$, then $\#\cA_{\bfn_1,\bfn_2,p_1,p_2}\ll p_1^{-1}W^{-3}H$.\end{lemma}
\begin{proof}For a prime $p>W$, $p$ is coprime to $q\in (\frac W2, W]$. The arithmetic progression $\cA_{\bfn,p}$ from \eqref{EqIntersection1} is bounded in length by \begin{equation}\label{EqLemIntersectionSize}\#\cA_{\bfn,p}\leq q^{-1}p^{-1}|I_{k}|\leq 2p^{-1}W^{-1}W^{-2}H=2p^{-1}W^{-3}H.\end{equation}

The lemma follows because $\cA_{\bfn_1,\bfn_2,p_1,p_2}=\cA_{\bfn_1,p_1}\cap\cA_{\bfn_2,p_2}$.\end{proof}

We remark that, on the other hand, if $H\geq 4pW^3$, then we also have \begin{equation}\label{EqLemIntersectionSize2}\#\cA_{\bfn,p}\geq q^{-1}(p^{-1}|I_{k}|-1)-1\geq \frac12 q^{-1}p^{-1}|I_{k}|\geq \frac12p^{-1}W^{-3}H.\end{equation}

We first take the sum when the length of $\cA_{\bfn_1,\bfn_2,p_1,p_2}$ is bounded by $2^{-s}W^{-(\ref{BONSTShortSeq}+3)}H$ where $\lBONST{BONSTShortSeq}\geq 10$ and will be determined later. This part of \eqref{EqMinorS2} is easily bounded as below.
\begin{proposition}\label{PropMinorShort}For $\ref{BONSTShortSeq}\geq 10$, the expression\begin{equation}\label{EqPropMinorShort}\begin{aligned}
&\sum_{\bfn_1,\bfn_2\in\cN\times\cJ}\sum_{\substack{p_1,p_2\in(2^{s-1},2^s]\\ \#\cA_{\bfn_1,\bfn_2,p_1,p_2}<2^{-s}W^{-(\ref{BONSTShortSeq}+3)}H}}\sum_{l\in\cA_{\bfn_1,\bfn_2,p_1,p_2}}\beta(p_1)\overline{\beta(p_2)}\\
&\hskip2cm \tF_{\bfn_1}(g_{\bfn_1}(p_1l)\Gamma_{\bfn_1})\overline{\tF_{\bfn_2}(g_{\bfn_2}(p_2l)\Gamma_{\bfn_2}))}\end{aligned}\end{equation}
satisfies $|\eqref{EqPropMinorShort}|\ll 2^sW^{-\ref{BONSTShortSeq}}H^2 N.$\end{proposition}
\begin{proof}
$$\begin{aligned}
|\eqref{EqPropMinorShort}|\leq &\bigg|\sum_{\bfn_1,\bfn_2\in\cN\times\cJ}\sum_{\substack{p_1,p_2\in(2^{s-1},2^s]\\ \#\cA_{\bfn_1,\bfn_2,p_1,p_2}<2^{-s}W^{-(\ref{BONSTShortSeq}+3)}H}}2^{-s}W^{-(\ref{BONSTShortSeq}+3)}H\bigg|\\
\leq&2^{-s}W^{-(\ref{BONSTShortSeq}+3)}H\sum_{p_1,p_2\in(2^{s-1},2^s]}\sum_{\bfn_1,\bfn_2\in\cN\times\cJ}1_{\cA_{\bfn_1,\bfn_2,p_1,p_2}\neq\emptyset}
\\
\ll& 2^{-s}W^{-(\ref{BONSTShortSeq}+3)}H\cdot 2^{2s}\cdot  W^3 N \cdot  H=2^sW^{-\ref{BONSTShortSeq}}H^2 N
\end{aligned}$$
\end{proof}

Here the last inequality follows from \eqref{EqEnsembleSize} and the lemma below.
\begin{lemma}\label{LemIntersect}If $2^s\geq W\geq 10$, then for all $\bfn_1\in\cN\times\cJ$ and $p_1,p_2\in(2^{s-1},2^s]$,
$$\#\{\bfn_2\in\cN\times\cJ:\ \cA_{\bfn_1,\bfn_2,p_1,p_2}\neq\emptyset\}\ll  H.$$
\end{lemma}
\begin{proof}Notice that if in $\bfn_2=(n_2, k_2, j_2)$, $k_2$ is given, then $\cA_{\bfn_1,\bfn_2,p_1,p_2}\neq\emptyset$ implies $(\frac{n_1}{p_1}+\frac{1}{p_1}I_{k_1})\cap(\frac{n_2}{p_2}+\frac{1}{p_2}I_{k_2})\neq \emptyset$. This is true only if $n_2$ belongs to an interval whose length is at most $$\frac{p_2}{p_1}|I_{k_1}|+|I_{k_2}|\leq 2W^{-2}H+W^{-2}H=3 W^{-2}H.$$

Moreover, the congruence class of elements in $\cA_{\bfn_1,p_1}$ modulo $q$ is determined by $\bfn_1$ and $p_1$. This congruence class, together with $n_2$ and $p_2$, in turn determines a unique choice of the remainder $j_2$ modulo $q$ in order for $\cA_{\bfn_1,\bfn_2,p_1,p_2}=\cA_{\bfn_1,p_1}\cap \cA_{\bfn_2,p_2}$.

Therefore, $ \sum_{\bfn_2\in\cN\times\cJ}1_{\cA_{\bfn_1,\bfn_2,p_1,p_2}\neq\emptyset}\ll\sum_{k_2\leq W^2} W^{-2}H= H.$
\end{proof}

We now focus on intersections with $\#\cA_{\bfn_1,\bfn_2,p_1,p_2}\geq 2^{-s}W^{-(\ref{BONSTShortSeq}+3)} H$.
\begin{definition}\label{DefPropJoining}For $s\in[s_-,s_+]$, $\bfn_1\in\cN\times\cJ$, prime number $p_1\in(2^{s-1},2^s]$ and a parameter $\ref{BONSTShortSeq}\geq 10$, the set $\Omega_{s,\bfn_1,p_1,\ref{BONSTShortSeq}}$ is defined to be the set of all configurations $(\bfn_2,p_2)\in\cN\times\cJ\times(2^{s-1},2^s]$ such that: \begin{enumerate}[(i)]
\item $p_2$ is prime;
\item\label{ItemPropJoining2} $\#\cA_{\bfn_1,\bfn_2,p_1,p_2}\geq 2^{-s}W^{-(\ref{BONSTShortSeq}+3)} H$;
\item\label{ItemPropJoining3} $$\begin{aligned}\Big|\sum_{l\in\cA_{\bfn_1,\bfn_2,p_1,p_2}}\tF_{\bfn_1}(g_{\bfn_1}(p_1l-n_1)\Gamma_{\bfn_1})\overline{\tF_{\bfn_2}(g_{\bfn_2}(p_2l-n_1)\Gamma_{\bfn_2})}
    \Big|\\
    \geq W^{-\ref{BONSTShortSeq}}\#\cA_{\bfn_1,\bfn_2,p_1,p_2}.\end{aligned}$$
\end{enumerate}
\end{definition}

\begin{proposition}\label{PropJoining}One can choose the constant $\ref{CONSTImplicit}=O(1)\geq 10$ to be sufficiently large, such that: if \begin{equation}\label{EqPropJoiningCond}W\geq  10, \ref{BONSTShortSeq}\geq 10, \ref{BONSTTotalED}\geq \ref{CONSTImplicit}\ref{BONSTShortSeq}, H\geq \max(W^{\ref{BONSTTotalED}},2^{10s}),\end{equation} then for all pairs $(\bfn_1,p_1)$, where $\bfn_1\subset\cN\times\cJ$ and $p_1\in(2^{s-1},2^s]$, $$\#\Omega_{s,\bfn_1,p_1,\ref{BONSTShortSeq}}< 2^sW^{-\ref{BONSTShortSeq}}H.$$\end{proposition}

The proof of the proposition is postponed to the next section.

\begin{proposition}\label{PropMinorLong}In the settings of Proposition \ref{PropJoining}, the expression \begin{equation}\label{EqPropMinorLong}\begin{aligned}
&\sum_{\bfn_1,\bfn_2\in\cN\times\cJ}\sum_{\substack{p_1,p_2\in(2^{s-1},2^s]\\ \#\cA_{\bfn_1,\bfn_2,p_1,p_2}\geq 2^{-s}W^{-(\ref{BONSTShortSeq}+3)} H}}\sum_{l\in\cA_{\bfn_1,\bfn_2,p_1,p_2}}\beta(p_1)\overline{\beta(p_2)}\\
&\hskip2cm \tF_{\bfn_1}(g_{\bfn_1}(p_1l-n_1)\Gamma_{\bfn_1})\overline{\tF_{\bfn_2}(g_{\bfn_2}(p_2l-n_2)\Gamma_{\bfn_2})}
\end{aligned}\end{equation} satisfies $|\eqref{EqPropMinorLong}|\ll 2^sW^{-\ref{BONSTShortSeq}}H^2N$.\end{proposition}

\begin{proof} As $|\beta|\leq 1$ and $\|\tF_{\bfn}\|_{C^0}\leq 2$ for all $\bfn$, in $|\eqref{EqPropMinorLong}|$, using Lemma \ref{LemIntersectionSize} and Proposition \ref{PropJoining}, the contribution from configuration with $(\bfn_2,p_2)\in \Omega_{s,\bfn_1,p_1,\ref{BONSTShortSeq}}$ is bounded by
\begin{equation}\label{EqPropMinorLong1}\begin{aligned}&(\#\cN\cdot\#\cJ)\cdot 2^s\cdot(\max_{\bfn_1,p_1}\#\Omega_{s,\bfn_1,p_1,\ref{BONSTShortSeq}})(\max_{\bfn_1,\bfn_2,p_1,p_2}\#\cA_{\bfn_1,\bfn_2,p_1,p_2})\cdot 4\\
\ll &N W^3\cdot 2^s \cdot 2^sW^{-\ref{BONSTShortSeq}}H\cdot 2p^{-1}W^{-3}H\\
\ll& 2^sW^{-\ref{BONSTShortSeq}}H^2N.\end{aligned}\end{equation}

And the contribution from out of this collection is bounded, thanks to Lemma \ref{LemIntersectionSize}, Lemma \ref{LemIntersect} and the construction of $\Omega_{s,\bfn_1,p_1,\ref{BONSTShortSeq}}$, by
\begin{equation}\label{EqPropMinorLong1}\begin{aligned}&(\#\cN\cdot\#\cJ)\cdot 2^{2s}\cdot\max_{\bfn_1,p_1,p_2}\sum_{\substack{\bfn_2\in\cN\times\cJ\\\cA_{\bfn_1,\bfn_2,p_1,p_2}\neq\emptyset}}\\
&\hskip5mm\Big|\sum_{l\in\cA_{\bfn_1,\bfn_2,p_1,p_2}}\tF_{\bfn_1}(g_{\bfn_1}(p_1l-n_1)\Gamma_{\bfn_1})\overline{\tF_{\bfn_2}(g_{\bfn_2}(p_2l-n_1)\Gamma_{\bfn_2})}
    \Big|\\
\ll& NW^3\cdot 2^{2s}\cdot H\cdot W^{-\ref{BONSTShortSeq}}\max_{\bfn_1,\bfn_2,p_1,p_2}\#\cA_{\bfn_1,\bfn_2,p_1,p_2}\\
\ll & NW^3\cdot 2^{2s}\cdot H\cdot W^{-\ref{BONSTShortSeq}}2^{-s}W^{-3}H\\
=&2^sW^{-\ref{BONSTShortSeq}}H^2N.
\end{aligned}\end{equation}

The lemma follows by combining these two bounds.
\end{proof}

Now adding up the estimates from Propositions \ref{PropMinorShort} and \ref{PropMinorLong} leads to the proof of Proposition \ref{PropMinor}.

\begin{proof}[Proof of Proposition \ref{PropMinor}]By Propositions \ref{PropMinorShort} and \ref{PropMinorLong}, when $\ref{CONSTImplicit}$ is sufficiently large, under assumptions \eqref{EqPropJoiningCond}, we have
\begin{equation}\begin{aligned}
\eqref{EqMinorS1}\ll &2^{-\frac s2}N^\frac12\eqref{EqMinorS2}^\frac12\leq 2^{-\frac s2}N^\frac12(\eqref{EqPropMinorShort}+\eqref{EqPropMinorLong})^\frac12\\
\ll&2^{-\frac s2}N^\frac12\cdot 2^{\frac s2}W^{-\frac{\ref{BONSTShortSeq}}2}HN^\frac12\\
=&W^{-\frac{\ref{BONSTShortSeq}}2}HN.
\end{aligned}\end{equation}

Hence,
\begin{equation}\begin{aligned}
|\eqref{EqCorSqFree}|= &|\eqref{EqMinorSDya}|\leq \sum_{s\in(s_-,s_+]}\eqref{EqMinorS1}\leq s_+W^{-\frac{\ref{BONSTShortSeq}}2}HN,
\end{aligned}\end{equation}
and by Corollary \ref{EqCorSqFree},
\begin{equation}\label{EqPropMinorPf1}\begin{aligned}
|\eqref{EqMinorS}|\leq  &|\eqref{EqCorSqFree}|+(2^{-s_-}+W^{-\ref{BONSTTotalED}})HN\\
\ll&(s_+W^{-\frac{\ref{BONSTShortSeq}}2}+2^{-s_-}+W^{-\ref{BONSTTotalED}})HN.
\end{aligned}\end{equation}

We now set the parameters $s_-$, $s_+$, $\ref{BONSTTotalED}$ and $\ref{BONSTShortSeq}$. Let $s_+=\lfloor \frac1{10}\log H\rfloor$. and $s_-=\lfloor20\epsilon s_+\rfloor$. This guarantees that $
N-\#\cS\ll\frac{s_-}{s_+}N\leq\epsilon N$. Moreover, $2^{-s_-}< H^{-\epsilon}$.

Assume in addition that $\ref{BONSTTotalED}\geq 10\ref{CONSTImplicit}$ and let $B_2=\ref{CONSTImplicit}^{-1}B_1$. The inequalities in \eqref{EqPropMinorCond}, together with the fact that $W\in[R,R^{\ref{CONSTWRange}\ref{BONSTTotalED}^m}]$, imply $W^{\ref{BONSTTotalED}}<R^{\ref{CONSTWRange}\ref{BONSTTotalED}^{m+1}}<H^\epsilon<H$. This also implies for all $s\in(s_-,s_+)$, $2^s>2^{s_-}>H^\epsilon>W$. So all conditions in \eqref{EqPropJoiningCond} are verified.

\eqref{EqPropMinorPf1} now yields
\begin{equation}\label{EqPropMinorPf2}\begin{aligned}
|\eqref{EqMinorS}|\ll  &(W^{-\frac{\ref{CONSTImplicit}^{-1}\ref{BONSTTotalED}}2}\log H+H^{-\epsilon}+W^{-\ref{BONSTTotalED}})HN\\
\ll& (W^{-\frac{\ref{CONSTImplicit}^{-1}\ref{BONSTTotalED}}2}\log H+H^{-\epsilon})HN.
\end{aligned}\end{equation} Finally, to complete the proof, one only needs to replace the value of the constant $\ref{CONSTImplicit}$ with $10\ref{CONSTImplicit}$.
\end{proof}

\section{Proof of  Proposition \ref{PropJoining}}\label{SecMinorII}

This part contains the proof of Proposition \ref{PropJoining} by contradiction. In the rest of Section \ref{SecMinorII}, we will assume that $t$, $s$, $\bfn_1$, $p_1$ are all fixed. For brevity, we will replace the notations $\bfn_2$ and $p_2$ with $\bfn$ and $p$.

Because one may choose the constant $\ref{CONSTImplicit}$ as long as it depends only on $m$ and $d$, instead of \eqref{EqPropJoiningCond} we will assume instead:
\begin{equation}\label{EqPropJoiningCond1}2^s>W\geq 10, \ref{BONSTShortSeq}\geq 10, \ref{BONSTTotalED}\geq 10\ref{CONSTImplicit}^2\ref{BONSTShortSeq}, H\geq \max(W^{\ref{BONSTTotalED}},2^{10s}),\end{equation}


In order to get contradiction, suppose for $\bfn_1\in\cN\times\cJ$ and $p_1\in(2^{s-1},2^s]$, \begin{equation}\label{EqPropJoiningCount}\#\Omega_{s,\bfn_1,p_1,\ref{BONSTShortSeq}}\geq 2^sW^{-\ref{BONSTShortSeq}}H.\end{equation} Let $(\bfn ,p )$ be an element of $\Omega_{s,\bfn_1,p_1,\ref{BONSTShortSeq}}$, then $p_1,p\geq 2^s>W \geq q$. By the proof of Lemma \ref{LemIntersectionSize}, as $\cA_{\bfn_1,\bfn ,p_1,p }$ is the intersection of two finite arithmetic progressions $\cA_{\bfn_1,p_1}$, $\cA_{\bfn ,p }$ of step length $q$, it also has step length $q$ itself whenever it is non-empty.

Since $\bfn_1$ and $p_1$ are fixed, the arithmetic progression $\cA_{\bfn_1,p_1}$ can be parametrized as $\{qt+r: t\in[T]\}$ for some $r\in\bZ$. Here  by \eqref{EqLemIntersectionSize} \begin{equation}\label{EqTLength}T=\#\cA_{\bfn_1,p_1}\leq 4\cdot2^{-s}W^{-3}H.\end{equation}

When $(\bfn ,p )\in \Omega_{s,\bfn_1,p_1,\ref{BONSTShortSeq}}$, the subsequence $\cA_{\bfn_1,\bfn ,p_1,p }$ has the form $\{qt+r: t\in\cA'_{\bfn ,p }\}$ where $\cA'_{\bfn ,p }$ is a subinterval of integers in $[T]$ of length $\#\cA _{\bfn_1,\bfn ,p_1,p }\geq 2^{-s}W^{-\ref{BONSTShortSeq}} H$.

The conditions \eqref{ItemPropJoining2} and \eqref{ItemPropJoining3} on $\Omega_{s,\bfn_1,p_1,\ref{BONSTShortSeq}}$ in Definition \ref{DefPropJoining} can be rewritten as
\begin{equation}\label{EqPropJoining1}
  \cA'_{\bfn ,p }\geq 2^{-s}W^{-(\ref{BONSTShortSeq}+3)} H
\end{equation}
and
\begin{equation}\label{EqPropJoining2}\begin{aligned}
 & \left|\sum_{t\in \cA'_{\bfn ,p }}\tF_{\bfn_1}(g_{\bfn_1}(p_1(qt+r)-n_1)\Gamma_{\bfn_1})\overline{\tF_\bfn(g_\bfn(p (qt+r)-n )\Gamma_\bfn)}
    \right| \\
\geq& W^{-\ref{BONSTShortSeq}}\#\cA'_{\bfn ,p }\end{aligned}
\end{equation}

For every configuration $(\bfn ,p )=(n ,\bfj ,p )=(n ,k ,j ,p )\in\Omega_{s,\bfn_1,p_1,\ref{BONSTShortSeq}}$. Define polynomial sequences $g_{\bfn ,p },\tg_{\bfn ,p }: \bZ\rightarrow G_{\bfn_1}\times G_\bfn$ by \begin{equation}\label{EqJoiningPoly}g_{\bfn ,p }(l)=\big(g_{\bfn_1}(p_1l-n_1),g_\bfn(p l-n)\big);\ \tg_{\bfn ,p }(t)=g_{\bfn ,p }(qt+r).\end{equation} Note that the definition of $\tg_{\bfn ,p }$ depends on the choice of $\bfn $.

Then $g_{\bfn ,p },\tg_{\bfn ,p }\in \Poly(\bZ, (G_{\bfn_1})_\bullet\times (G_\bfn)_\bullet)$. From \eqref{EqMinorTermC0Bd}, \eqref{EqTLength}, \eqref{EqPropJoining1} and \eqref{EqPropJoining2}, we know that the sequence $(\tg_{\bfn ,p }(t)(\Gamma\times\Gamma))_{t\in\cA'_{\bfn,p}}$ is not totally $2^{-2}W^{-\ref{BONSTShortSeq}}$-equidistributed in $(G_{\bfn_1}/\Gamma_{\bfn_1})\times(G_\bfn/\Gamma_\bfn)$. Then by Lemma \ref{LemEDTotalED}, for a shorter length $T'_{\bfn ,p }\geq 2^{-5}W^{-2\ref{BONSTShortSeq}}T$,  the sequence $(\tg_{\bfn ,p }(t)(\Gamma\times\Gamma))_{t\in[T'_{\bfn ,p }]}$ fails to be $2^{-5}W^{-2\ref{BONSTShortSeq}}$-equidistributed in $(G_{\bfn_1}/\Gamma_{\bfn_1})\times(G_\bfn/\Gamma_\bfn)$.

By Proposition \ref{PropLeibmanGT}, there exists a horizontal character $\eta_{\bfn ,p }$ of $(G_{\bfn_1}/\Gamma_{\bfn_1})\times(G_\bfn/\Gamma_\bfn)$ such that \begin{equation}\label{EqEta1}0<|\eta_{\bfn ,p }|<W^{O(\ref{BONSTShortSeq})}\end{equation} and $\|\eta_{\bfn ,p }\circ \tg_{\bfn ,p }\|_{C^\infty([T'_{\bfn ,p }])}\ll W^{O(\ref{BONSTShortSeq})}$. As $T'_{\bfn ,p }\gg W^{-2\ref{BONSTShortSeq}}T$, this implies that  \begin{equation}\label{EqEta2}\|\eta_{\bfn ,p }\circ \tg_{\bfn ,p }\|_{C^\infty([T])}\ll W^{O(\ref{BONSTShortSeq})}.\end{equation}  Here the norm $|\eta_{\bfn ,p }|$ is measured in terms of the Mal'cev basis $\cV_\bfn\cup\cV_{\bfn'}$, where $\cV_\bfn=\cV_{n ,\bfj }$ and $\cV_{\bfn_1}=\cV_{n_1,\bfj_1}$ are defined in Section \ref{SecBilinear}.

Recall from our construction in Section \ref{SecBilinear} that the sequences $G_\bfn$, $\Gamma_\bfn$, $\cV_\bfn$ are determined by $\gamma_\bfn$, which in turn depends only on the variables $n$, $j$ in $\bfn=(n,k,j)$ and is $q$-periodic in $n$. So there are $\gamma_*$, $G_*$, $\Gamma_*$, $\cV_*$ such that for at least $q^{-2}\#\Omega_{s.\bfn_1,p_1}$ choices of $(\bfn ,p )\in\Omega_{s,\bfn_1,p_1,\ref{BONSTShortSeq}}$, \begin{equation}\label{EqPigeonG}(\gamma_\bfn, G_\bfn,\Gamma_\bfn, \cV_\bfn)=(\gamma_*,G_*,\Gamma_*, \cV_*).\end{equation} Note that the choices of horizontal characters satisfying \eqref{EqEta1} is bounded by $W^{O(\ref{BONSTShortSeq})}$. Given \eqref{EqPropJoiningCount} and that $q\leq W$, by pigeonhole principle, we can find some horizontal character $\eta$ of $(G_{\bfn_1}/\Gamma_{\bfn_1})\times(G_*/\Gamma_*)$ such that for a set $\Omega_*$ of at least $2^sW^{-O(\ref{BONSTShortSeq})}H$ choices of $(\bfn ,p)\in\Omega_{s,\bfn_1,p_1,\ref{BONSTShortSeq}}$, \eqref{EqPigeonG} holds and $\eta_{\bfn ,p }=\eta$.

Therefore,
\begin{align}\label{EqPigeonNorm}\|\eta\circ  \tg_{\bfn ,p }\|_{C^\infty[T]}\ll W^{O(\ref{BONSTShortSeq})}\end{align} holds for at least $2^sW^{-O(\ref{BONSTShortSeq})}H$ choices of $(\bfn ,p)\in\Omega_{s,\bfn_1,p_1,\ref{BONSTShortSeq}}$. In particular, because of the fact $\#\cJ\leq W^3$ and Lemma \ref{LemIntersect}, there is a set $\cP_{s,\bfn_1,p_1}\subseteq\{p\text{ prime:} p\in(2^{s-1},2^s]\}$ of size \begin{equation}\label{EqPSetSize}\#\cP_{s,\bfn_1,p_1}\gg 2^sW^{-O(\ref{BONSTShortSeq})},\end{equation} such that for all $p\in\cP_{s,\bfn_1,p_1}$, there are at least $W^{-O(\ref{BONSTShortSeq})}H$ choices of $n$, such that for some $\bfj$, the configuration $\bfn=(n,\bfj)$ satisfies $(\bfn,p)\in\Omega_{s,\bfn_1,p,\ref{BONSTShortSeq}}$ and \eqref{EqPigeonNorm}.

Recall that $g_\bfn(h)=\gamma_\bfn^{-1}g'(n,h)\gamma_\bfn$. So for the polynomial $g_*(n,h)=\gamma_*^{-1}g'(n,h)\gamma_*$ and every $(\bfn ,p )\in\Omega_{s,\bfn_1,p_1,\ref{BONSTShortSeq}}$, $g_\bfn(h)=g_*(n ,h)$ where $n $ is the first coordinate of $\bfn =(n ,k ,j )$. In this case,
\begin{equation}\label{EqJoiningPoly1}\tg_{\bfn,p}(t)=\big(g_{\bfn_1}(p_1(qt+r)-n_1),g_*(n,p(qt+r)-n)\big).\end{equation}

Write $\eta=\eta_{(1)}\oplus\eta_{(2)},$ where $\eta_{(1)}$ and $\eta_{(2)}$ are respectively horizontal characters of $G_{\bfn_1}/\Gamma_{\bfn_1}$ and $G_*/\Gamma_*$ and at least one of them is non-zero. Then $\eta_{(1)}\circ g_{\bfn_1}:\bZ\to\bR$ and $\eta_{(1)}\circ g_*:\bZ^2\to\bR$ are polynomials of total degree bounded by $d$, where $d$ is the step of nilpotency of $G_\bullet$. As $p_1$, $r$, $q$, $\bfn_1$ are all fixed, one can write
\begin{equation}\label{EqEta1}\eta_{(1)}\circ g_{\bfn_1}(t)=\sum_{l=0}^{d }\alpha_lt^l.\end{equation}
\begin{equation}\label{EqEta2_0}\eta_{(1)}\circ g_*(n,h)=\sum_{\substack{l_1,l_2\geq 0\\l_1+l_2\leq d }}\beta^*_{l_1,l_2}n^{l_1}h^{l_2}.\end{equation}

We now parametrize $\eta_{(2)}\circ g_*$ in a better way. When $(\bfn,p)\in\Omega_{s,\bfn_1,p_1,\ref{BONSTShortSeq}}$, $\cA_{\bfn_1,n,p_1,p}\neq\emptyset$. So we can fix an $t_0=t_0(\bfn,p)\in[T]$ such that $p(qt_0+r)-n\in\cI_\bfn\subset [H]$. On the other hand, because $t_0\leq T=\#\cA_{\bfn_1,p_1}$, by \eqref{EqLemIntersectionSize}, $0<pqt_0\leq 2pq\cdot q^{-1}p_1^{-1}W^{-2}H\leq 4W^{-2}H$. Thus $pr-n\in[-4W^{-2}H,H]\subseteq (-H,H].$ We will write $b=pr-n+H$. Then $b\in[2H]$. For $u\in\bZ$, we can write

\begin{equation}\label{EqEta2U}\begin{aligned}&\eta_{(2)}\circ g_*(n,qu+pr-n)\\
=&\eta_{(2)}\circ g_*(pr+H-b,qu+b-H)\\
=&\sum_{\substack{l_1,l_2\geq 0\\l_1+l_2\leq d }}\beta^*_{l_1,l_2}(pr+H-b)^{l_1}(qu+b-H)^{l_2}\\
=:&\sum_{\substack{l_1,l_2,i\geq 0\\l_1+l_2+i\leq d }}\beta_{l_1,l_2,i}p^{l_1}u^{l_2}b^i\end{aligned}\end{equation}

In particular, for $u=pt$, we have
\begin{equation}\label{EqEta2}\begin{aligned}&\eta_{(2)}\circ g_*(n,p(qt+r)-n)\\
=&\eta_{(2)}\circ g_*(pr+H-b,q(pt)+b-H)\\
=&\sum_{\substack{l_1,l_2,i\geq 0\\l_1+l_2+i\leq d }}\beta_{l_1,l_2,i}p^{l_1}(pt)^{l_2}b^i\\
=&\sum_{l=0}^{d }\sum_{l'=l}^{d }\sum_{i=0}^{d -l'}\beta_{l'-l,l,i}p^{l'}b^it^l\end{aligned}\end{equation}
then
\begin{equation}\label{EqJoiningPoly2}\eta\circ \tg_{\bfn ,p}(t)=\sum_{l=0}^{d }(\alpha_l+\sum_{l'=l}^{d }\sum_{i=0}^{d -l'}\beta_{l'-l,l,i}p^{l'}b^i)t^l,\end{equation} where the coefficients $\beta_{l'-l,l,i}$ are independent of $p$,$b$ and $t$ (but depend on $\bfn_1$, $p_1$ and $H$).

The earlier discussion asserts that for all $p\in\cP_{s,\bfn_1,p_1}$, there are is a subset $\cB_{s,\bfn_1,p_1,p}\subseteq[2H]$ whose size satisfies \begin{equation}\label{EqBSetSize}\#\cB_{s,\bfn_1,p_1,p}\gg W^{-O(\ref{BONSTShortSeq})}H\end{equation} such that for all $b\in\cB_{s,\bfn_1,p_1,p}$, $\|\eqref{EqJoiningPoly2}\ (\mod \bZ)\|_{C^\infty([T])}\ll W^{O(\ref{BONSTShortSeq})}$, where \eqref{EqJoiningPoly2} is regarded as a polynomial in $t$.

For such pairs $(p,b)$, by Lemma \ref{LemCoeffBound} and \eqref{EqTLength}, we can find a positive integer $Z_1\ll O(1)$ such that for all $0\leq l\leq d$,
\begin{equation}\label{EqDivisor1}\Big\|Z_1(\alpha_l+\sum_{l'=l}^{d }\sum_{i=0}^{d -l'}\beta_{l'-l,l,i}p^{l'}b^i)\Big\|_{\bR/\bZ}
\ll  W^{O(\ref{BONSTShortSeq})}T^{-l}\ll 2^{ls}W^{O(\ref{BONSTShortSeq})}H^{-l}.\end{equation}
By using pigeonhole principle, one can make $Z_1$ independent of $b$ after substituting $\cB_{s,\bfn_1,p_1,p}$ with a smaller subset whose size still satisfies the lower bound \eqref{EqPSetSize}.

We now view $Z_1(\alpha_l+\sum_{l'=l}^{d }\sum_{i=0}^{d -l'}\beta_{l'-l,l,i}p^{l'}b^i)$ as a polynomial of $b$. Applying Lemma \ref{LemCNormBound} (with $\epsilon= 2^{ls}W^{O(\ref{BONSTShortSeq})}H^{-l}$ and $\delta=W^{-O(\ref{BONSTShortSeq})}$), we reduce from \eqref{EqDivisor1} that there is a positive integer $Z_2\ll W^{O(\ref{BONSTShortSeq})}$ such that
\begin{equation}\label{EqDivisor2}\Big\|Z_2Z_1(\alpha_l+\sum_{l'=l}^{d }\sum_{i=0}^{d -l'}\beta_{l'-l,l,i}p^{l'}b^i)\ (\mod \bZ)\Big\|_{C^\infty[2H]}\ll 2^{ls}W^{O(\ref{BONSTShortSeq})}H^{-l},\end{equation}

Again by Lemma \ref{LemCoeffBound},  for all $p\in\cP_{s,\bfn_1,p_1}$, there is a positive integer $Z_3\ll O(1)$, such that for all $i\geq 1$, $l\geq 0$ such that $i+l\leq d$,
\begin{equation}\label{EqDivisor3}\Big\|Z_3Z_2Z_1\sum_{l'=l}^{d -i}\beta_{l'-l,l,i}p^{l'}\Big\|_{\bR/\bZ}
\ll 2^{ls}W^{O(\ref{BONSTShortSeq})}H^{-i-l}.\end{equation}
And, when $i=0$, for all $0\leq l\leq d$, \begin{equation}\label{EqDivisor3'}\Big\|Z_3Z_2Z_1(\alpha_l+\sum_{l'=l}^d\beta_{l'-l,l,0}p^{l'})\Big\|_{\bR/\bZ}
\ll 2^{ls}W^{O(\ref{BONSTShortSeq})}H^{-l}.\end{equation}

Lemma \ref{LemCNormBound} applies again, with respect to the variable $p\in[2^s]$, with $\epsilon=2^{ls}W^{O(\ref{BONSTShortSeq})}H^{-l}$, $\delta=W^{-O(\ref{BONSTShortSeq})}$,  and yields a positive integer $Z_4\ll W^{O(\ref{BONSTShortSeq})}$ that:

For all $i\geq 1$, $0\leq l\leq d$ subject to $i+l'\leq d$,
\begin{equation}\label{EqDivisor4}\Big\|Z_4Z_3Z_2Z_1\sum_{l'=l}^{d -i}\beta_{l'-l,l,i}p^{l'}\ (\mod \bZ)\Big\|_{C^\infty([2^s])}
\ll 2^{ls}W^{O(\ref{BONSTShortSeq})}H^{-i-l};\end{equation}
and for $i=0$ and $0\leq l\leq d$,
\begin{equation}\label{EqDivisor4'}\Big\|Z_4Z_3Z_2Z_1(\alpha_l+\sum_{l'=l}^d\beta_{l,l',0}p^{l'})\ (\mod \bZ)\Big\|_{C^\infty([2^s])}
\ll 2^{ls}W^{O(\ref{BONSTShortSeq})}H^{-l}.\end{equation}

A final round of application of Lemma \ref{LemCoeffBound} tells us that, for a positive integer $Z_5\ll O(1)$, the following properties hold:

For all $i\geq 1$, $0\leq l\leq l'\leq d$ subject to $i+l\leq d$,
\begin{equation}\label{EqDivisor5}\Big\|Z_5Z_4Z_3Z_2Z_1\beta_{l'-l,l,i}\Big\|_{\bR/\bZ}
\ll 2^{(l-l')s}W^{O(\ref{BONSTShortSeq})}H^{-i-l};\end{equation}
in addition, for $i=0$ and $0\leq l\leq l'\leq d$ with $l'\geq 1$, \eqref{EqDivisor5} also holds.

Write $Z=Z_5Z_4Z_3Z_2Z_1$, which is an integer that is independent of $b$ and $t$, and satisfies $Z\ll W^{O(\ref{BONSTShortSeq})}$. Thus the character $Z\eta_{(2)}$ satisfies \begin{equation}\label{EqDivisorSize}|Z\eta_{(2)}|\ll |Z|\cdot|\eta|\ll W^{O(\ref{BONSTShortSeq})}.\end{equation}

As we state in Notation \ref{NotationMain}, one choose a sufficiently large constant $\ref{CONSTImplicit}=O(1)\geq 10$ which serves as the implicit constants both in the exponent of $W^{O(\ref{BONSTShortSeq})}$ of \eqref{EqDivisor5} and in \eqref{EqDivisorSize}. Now \eqref{EqDivisor5} writes \begin{equation}\label{EqDivisorDioph6}\Big\|Z\beta_{l'-l,l,i}\Big\|_{\bR/\bZ}
\ll 2^{(l-l')s}W^{\ref{CONSTImplicit}\ref{BONSTShortSeq}}H^{-i-l};\end{equation}

In other words, the inequality \begin{equation}\label{EqDivisorDioph}\Big\|Z\beta_{l_1,l_2,i}\Big\|_{\bR/\bZ}
\ll 2^{-l_1s}W^{\ref{CONSTImplicit}\ref{BONSTShortSeq}}H^{-i-l_2}\end{equation} holds for all integer triples $(l_1,l_2,i)$ such that $l_1,l_2,i\geq 0$, $l_1+l_2+i\leq d$ and $l_1$, $l_2$, $i$ are not simultaneously equal to $0$.

\begin{lemma}\label{LemFailTotalED} One can choose the constant $\ref{CONSTImplicit}=\ref{CONSTImplicit}(m,d)\geq 10$ to be sufficiently large, such that :

If \eqref{EqPropJoiningCond1} and   \eqref{EqPropJoiningCount}  both hold then for every configuration $(\bfn,p)\in\Omega_{s,\bfn_1,p_1,\ref{BONSTShortSeq}}$,  the sequence $\{g_\bfn(h)\Gamma_\bfn\}_{h\in[H]}$ is not totally $W^{-\ref{CONSTImplicit}^{-1}\ref{BONSTTotalED}}$-equidistibuted in $G_\bfn/\Gamma_\bfn$.
\end{lemma}

\begin{proof} Let $r$ and $b$ be as above. Set $\cU_{\bfn,p}=\{u\in\bZ:qu+b-H\in[H]\}$. Then $\cU_{\bfn,p}$ is an interval of integers, whose length satisfies $\frac Hq-1<\#\cU_{\bfn,p}<\frac Hq+1$. Moreover, as $0<b\leq 2H$, every $u\in\cU_{\bfn,p}$ satisfies $|u|\leq\frac {2H}q$.

Fix any subinterval $\cU'_{\bfn,p}\subset\cU_{\bfn,p}$ of integers, that is of length $\lceil\frac{2W^{-2\ref{CONSTImplicit}\ref{BONSTShortSeq}-3}H}q\rceil$. We note that because of \eqref{EqPropJoiningCond1}, $\#\cU'_{\bfn,p}\geq 10$. Then for any $u_1,u_2\in\cU'$, by \eqref{EqEta2},
\begin{equation}\label{EqEt2Concentrate1}\begin{aligned}
&\|Z\eta_{(2)}\circ g_*(n,qu_1+b-H)-Z\eta_{(2)}\circ g_*(n,qu_2+b-H)\|_{\bR/\bZ}\\
=&\Big\|Z\sum_{\substack{l_1,l_2,i\geq 0\\l_1+l_2+i\leq d }}\beta_{l_1,l_2,i}p^{l_1}b^i(u_1^{l_2}-u_2^{l_2})\Big\|_{\bR/\bZ}\\
=&\Big\|Z\sum_{\substack{l_1,l_2,i\geq 0\\l_1+l_2+i\leq d }}\beta_{l_1,l_2,i}p^{l_1}b^i(u_1-u_2)\sum_{h=0}^{l_2-1}u_1^hu_2^{l_2-1-h}\Big\|_{\bR/\bZ}\\
\ll&\sum_{\substack{l_1,i\geq 0; l_2\geq 1\\l_1+l_2+i\leq d}} 2^{-l_1s}W^{\ref{CONSTImplicit}\ref{BONSTShortSeq}}H^{-i-l_2}\cdot (2^s)^{l_1} (2H)^i \Big(\frac{W^{-2\ref{CONSTImplicit}\ref{BONSTShortSeq}-3}H}q\Big)\Big(\frac Hq\Big)^{l_2-1}\\
=&\sum_{\substack{l_1,i\geq 0; l_2\geq 1\\l_1+l_2+i\leq d}} (W^{-\ref{CONSTImplicit}\ref{BONSTShortSeq}-3})q^{-l_2}\\
\ll& W^{-\ref{CONSTImplicit}\ref{BONSTShortSeq}}.
\end{aligned}\end{equation}

This implies that for the the mapping $\teta(x)=\exp(2\pi i Z\eta_{(2)}(x))$ from $G/\Gamma$ to the unit circle in $\bC$, the values of $\teta(g_\bfn(h))$ are within distance $\ll W^{-\ref{CONSTImplicit}\ref{BONSTShortSeq}}$ to each other for $h\in \{qu+b-H: u\in\cU'_{\bfn,p}\}$. Again, using the convention in Notation \ref{NotationMain}, one can assume that the implicit constant here is $\ref{CONSTImplicit}$. In particular,
\begin{equation}\begin{aligned}
\Big|\Exp_{h\in  \{qu+b-H: u\in\cU'_{\bfn,p}\}}\teta(g_\bfn(h)\Gamma_\bfn)\Big|>1-\ref{CONSTImplicit}W^{-\ref{CONSTImplicit}\ref{BONSTShortSeq}}\geq\frac12,\end{aligned}\end{equation}  as we assumed $\ref{CONSTImplicit}$, $\ref{BONSTShortSeq}$ and $W$ are all bounded by $10$ from below. Because $Z\eta$ is a non-zero character, $\teta$ has zero average on $G_\bfn/\Gamma_\bfn$. In addition, $\|\teta\|_{G_\bfn/\Gamma_\bfn}\ll|Z\eta_{(2)}|\leq W^{\ref{CONSTImplicit}\ref{BONSTShortSeq}}$.

Now note that $\{qu+b-H: u\in\cU'_{\bfn,p}\}\subseteq[H]$ is an arithmetic progression whose length is greater than $W^{-2\ref{CONSTImplicit}\ref{BONSTShortSeq}-4}H$. It follows that the sequence $\{g_\bfn(h)\Gamma_\bfn\}_{h\in [H]}$ is not totally $\min(W^{-2\ref{CONSTImplicit}\ref{BONSTShortSeq}-4},\frac12W^{-\ref{CONSTImplicit}\ref{BONSTShortSeq}})$-equidistributed in $G_\bfn/\Gamma_\bfn$.

To finish the proof of Lemma \ref{LemFailTotalED}, it suffices to notice that by the assumptions in \eqref{EqPropJoiningCond1}, $\min(W^{-2\ref{CONSTImplicit}\ref{BONSTShortSeq}-4},\frac12W^{-\ref{CONSTImplicit}\ref{BONSTShortSeq}})\geq W^{-\ref{CONSTImplicit}^{-1}\ref{BONSTTotalED}}$.\end{proof}

\begin{proof}[Proof of Proposition \ref{PropJoining}] Recall that after redefining $\ref{CONSTImplicit}$ we may assume \eqref{EqPropJoiningCond1} instead of \eqref{EqPropJoiningCond}. By Lemma \ref{LemFailTotalED}, and the construction of $\cN$ in Lemma \ref{LemGenericED}, if \eqref{EqPropJoiningCount} holds, then for all $\bfn\in\Omega_{s,\bfn_1,p_1,\ref{BONSTShortSeq}}$, $\bfn\notin\cN\times\cJ$. This contradicts the definition of $\Omega_{s,\bfn_1,p_1,\ref{BONSTShortSeq}}$, which requires $\bfn\in\cN\times\cJ$. Therefore, \eqref{EqPropJoiningCount} is false for all $\bfn_1\in\cN\times\cJ$; in other words, Proposition \ref{PropJoining} is true.
\end{proof}

\section{Proof of the main theorem}\label{SecFinal}

Theorem \ref{ThmMain} will follow from

\begin{theorem}\label{ThmMinorMajor} Suppose $G$ is a connected, simply connected nilpotent Lie group and $\Gamma\subset G$ is a lattice. Assume that there exists an $R_0$-rational Mal'cev basis $\cV$ of the Lie algebra $G$ adapted to a nilpotent filtration $G_\bullet$ and the lattice $\Gamma$. Then there are constants $C,\epsilon_0>0$ that only depend on the dimension $m$ of $G$, such that for all $g\in \Poly(\bZ^2,G_\bullet)$, $1$-bounded multiplicative function $\beta:\bN\to\bC$, and continuous function $F:G/\Gamma\to\bR$, $H,N\in\bN$, $\epsilon>0$, if \begin{equation}\label{EqMinorMajorCond}\max\Big(\frac{\log R_0}{\log H},\frac{\log\log H}{\log H}\Big)<\epsilon<\epsilon_0;\ \log H<(\log N)^{\frac12},.\end{equation} then
\begin{equation}\label{EqMinorMajor}\begin{aligned}&\frac1{HN}\sum_{n\leq N}\Big|\sum_{h\leq H} \beta(n+h)F(g(n,h)\Gamma)\Big|\\
\ll &\Big(H^{-\epsilon}+H^{C\epsilon}e^{-\frac12\tM(\beta,\frac N{H^{C\epsilon}}, H^{C\epsilon})} \tM(\beta,\frac N{H^{C\epsilon}}, H^{C\epsilon})^{\frac12}\\
&\hskip1cm + H^{C\epsilon}(\log\frac N{H^{C\epsilon}})^{-\frac1{100}}\Big)HN.\end{aligned}\end{equation}\end{theorem}

\begin{proof}
Let $\ref{BONSTTotalED}=\ref{CONSTImplicit}$, $\lCONST{CONSTMinorMajor}=\ref{CONSTWRange}\ref{BONSTTotalED}^{m+1}=O(1)$ and $R=H^{\ref{CONSTMinorMajor}^{-1}\epsilon'}$. Combining Propositions \ref{PropMajor} and \ref{PropMinor}, we know that if the following inequalities hold :
\begin{equation}\label{EqMinorMajorCond1} \frac{\log\log H}{\log H}<\epsilon'<\frac 1{500}; H^{\ref{CONSTMinorMajor}^{-1}\epsilon'}\geq R_0\geq 10;\ \log H<(\log N)^{\frac12}.\end{equation}
then there exists a subset $\cS\subseteq[0,N]\cap\bN$, determined by $H$, $N$, and $\epsilon'$, with $N-\#\cS\ll\epsilon' N$, such that
\begin{equation}\label{EqMinorMajor1}\begin{aligned}&\sum_{n\leq N}\Big|\sum_{h\leq H} \beta(n+h)F(g(n,h)\Gamma)\Big|\\
\ll & \Big(W^{-1}\log H+H^{-\epsilon'}+W^{-\frac14}\\
&\hskip1cm +W^2e^{-\frac12\tM(\beta,\frac N{W^5}, W)} \tM(\beta,\frac N{W^5}, W)^{\frac12}+ W^2(\log\frac N{W^5})^{-\frac1{100}}\Big)HN\\
\ll &\Big(H^{-\ref{CONSTMinorMajor}^{-1}\epsilon'}\log H+H^{2\epsilon'}e^{-\frac12\tM(\beta,\frac N{H^{5\epsilon'}}, H^{\epsilon'})} \tM(\beta,\frac N{H^{5\epsilon'}}, H^{\epsilon'})^{\frac12}\\
&\hskip1cm + H^{2\epsilon'}(\log\frac N{H^{5\epsilon'}})^{-\frac1{100}}\Big)HN.\end{aligned}\end{equation}
Here we used the fact that $W\in[R,R^{\ref{CONSTWRange}\ref{BONSTTotalED}^m}]\subseteq[H^{\ref{CONSTMinorMajor}^{-1}\epsilon'},H^{\epsilon'}]$, and that the function $\tM(\beta,\frac N{W^5}, W)$ is decreasing in $W$. The set $\cS$ is the union of both the exceptional sets from Propositions \ref{PropMajor} and \ref{PropMinor}.

We now rewrite $\epsilon=\frac12\ref{CONSTMinorMajor}^{-1}\epsilon'$ and assume $\epsilon>\frac{\log\log H}{\log H}$. Then $H^\epsilon>\log H$ and $$H^{-\ref{CONSTMinorMajor}^{-1}\epsilon'}\log H=H^{-2\epsilon}\log H<H^{-\epsilon}.$$ Note that \eqref{EqMinorMajorCond1} implies \eqref{EqMinorMajorCond}. So \eqref{EqMinorMajor1} becomes

\begin{equation}\label{EqMinorMajor2}\begin{aligned}&\sum_{n\leq N}\Big|\sum_{h\leq H} \beta(n+h)F(g(n,h)\Gamma)\Big|\\
\ll &\Big(H^{-\epsilon}+H^{4\ref{CONSTMinorMajor}\epsilon}e^{-\frac12\tM(\beta,\frac N{H^{10\ref{CONSTMinorMajor}\epsilon}}, H^{2\ref{CONSTMinorMajor}\epsilon})} \tM(\beta,\frac N{H^{10\ref{CONSTMinorMajor}\epsilon}}, H^{2\ref{CONSTMinorMajor}\epsilon})^{\frac12}\\
&\hskip1cm + H^{4\ref{CONSTMinorMajor}\epsilon}(\log\frac N{H^{10\ref{CONSTMinorMajor}\epsilon}})^{-\frac1{100}}\Big)HN.\end{aligned}\end{equation}

The theorem follows by letting $C=10\ref{CONSTMinorMajor}$ and $\epsilon_0=\frac1{1000\ref{CONSTMinorMajor}}$, which only on $m$ and $d$. But as $d\leq m$, the dependence on $d$ can be suppressed.
\end{proof}

\begin{proof}[Proof of Theorem \ref{ThmMain}] First choose $R_0\geq 10$ such that $\gog$ has an $R_0$-rational Mal'cev basis with respect to the lower central series filtration $G_\bullet$ and lattice $\Gamma$. We then fix $H_0$ such that $\log H_0\geq R_0$.

 Notice that $f(n,h)=g^{n+h}x\in G/\Gamma$ is a polynomial map from $\Poly(\bZ^2,G_\bullet)$. Furthermore, in \eqref{EqMinorMajorCond}, $\max\Big(\frac{\log R_0}{\log H},\frac{\log\log H}{\log H}\Big)=\frac{\log\log H}{\log H}$ for all $H> H_0$. Hence Theorem \ref{ThmMinorMajor} can be  applied. The output is \eqref{EqThmMainSSize} and \eqref{EqThmMainS}, with
$$\delta(a,N)=a^C e^{-\frac12\tM(\beta,\frac N{a^C}, a^C)} \tM(\beta,\frac N{a^C}, a^C)^{\frac12}+a^C(\log\frac N{a^C})^{-\frac1{100}}.$$ We need to show $\displaystyle\lim_{N\to\infty}\delta(a,N)=0$ for all $a>0$, which is equivalent to that \begin{equation}\label{EqNonPret}\lim_{X\to\infty}\tM(\beta,X,Y)=\infty,\ \forall Y>0.\end{equation}

When $\beta$ is the M\"obius function $\mu$ or the Liouville function $\lambda$, it is known that $\lim_{N\to\infty}\frac1X\sum_{n\leq X}\beta(n)\chi(n)=0$.
By Hal\'asz's Theorem \cite{H68},  for any given Dirichlet character $\chi$, $\displaystyle\lim_{X\to\infty}\bD(\beta\chi, 1, X)=\infty$. Moreover, \cite{MRT16}*{Lemma C.1}, which is based on an argument of Granville and Soundararajan \cite{GS07}, guarantees that
$$\inf_{|t|\leq X}\bD(\beta\chi, n^{it},X)\geq\frac14\min\big(\sqrt {\log\log X}, D(\beta\chi, 1, X)\big) +O(1).$$ Therefore, for all Dirichlet characters $\chi$, $M(\beta\chi,X)\to\infty$ as $X\to\infty$. This implies \eqref{EqNonPret} by construction \eqref{EqDefM1} of $\tM$.

Finally, it remains to show \eqref{EqThmMain}. To see this, it suffices to notice that, because  because $N>\exp((\log H)^2)=H^{\log H}>H\log H>H\epsilon^{-1}$,
$$\begin{aligned}&\frac1{HN}\left|\sum_{n=1}^N\Big|\sum_{l=n+1}^{n+H} 1_\cS\mu(l)F(g^lx)\Big|-\sum_{n=1}^N\Big|\sum_{l=n+1}^{n+H} \mu(l)F(g^lx)\Big|\right|\\
\leq &\frac1{HN}\left|\sum_{n=1}^N\#((n,n+H]\backslash\cS)\right|\leq \frac1{HN}\cdot H\#([N+H]\backslash\cS)\\
\ll &\frac1N(\epsilon N+H)\ll \epsilon.
\end{aligned}$$ So \eqref{EqThmMain} can be deduced from \eqref{EqThmMainS}.\end{proof}

\end{document}